\documentclass[11pt]{amsart}

\usepackage{color,graphicx,enumerate,wrapfig,amssymb}

\usepackage{subfigure, caption}

\usepackage{eucal}

\usepackage{setspace}

\usepackage{hyperref}

\usepackage{algorithm}
\usepackage[noend]{algpseudocode}

\makeatletter
\def\BState{\State\hskip-\ALG@thistlm}
\makeatother

\usepackage{graphicx}
\usepackage{transparent}

\usepackage{enumitem}

\newcommand\indentdisplays[1]{%
     \everydisplay{\addtolength\displayindent{#1}%
     \addtolength\displaywidth{-#1}}}  

\hypersetup{
    bookmarks=true,         
    unicode=true,          
    pdftoolbar=true,        
    pdfmenubar=true,        
    pdffitwindow=false,     
    pdfstartview={FitH},    
    pdftitle={My title},    
    pdfauthor={Author},     
    pdfsubject={Subject},   
    pdfcreator={Creator},   
    pdfproducer={Producer}, 
    pdfkeywords={keyword1} {key2} {key3}, 
    pdfnewwindow=true,      
    colorlinks=true,       
    linkcolor=blue,          
    citecolor=blue,        
    filecolor=magenta,      
    urlcolor=cyan           
}

\renewenvironment{proof}[1][\proofname ]{{\noindent \bfseries #1. }}{\qed \bigskip } 
\def\subjclass#1{{\renewcommand{\thefootnote}{}%
\footnote{\emph{Mathematics Subject Classification (2010):} #1}}}

\newcommand{\R}{{\mathbb R}}

\newcommand{\Z}{{\mathbb Z}}

\newcommand{\cH}{{\mathcal H}}

\newcommand{\e}{\varepsilon}

\newcommand{\supp}{\operatorname{supp}}

\floatname{algorithm}{Algorithm}

\newtheorem{theorem}{Theorem}[section]

\newtheorem{cor}[theorem]{Corollary}

\newtheorem{definition}{Definition}[section]

\newtheorem{lem}[theorem]{Lemma}

\newtheorem{prop}[theorem]{Proposition}

\newtheorem{remark}[theorem]{Remark}

\definecolor{orange}{RGB}{220,110,0}
\definecolor{violet}{RGB}{141,10,100}

\numberwithin{equation}{section}

\title[Perturbed sandpile]{Perturbed divisible sandpiles and Quadrature surfaces}

\author{Hayk Aleksanyan}
\address{Department of Mathematics, KTH Royal Institute of Technology, SE-100 44  Stockholm,
Sweden}
\email{hayk.aleksanyan@gmail.com}

\author{Henrik Shahgholian}
\address{Department of Mathematics, KTH Royal Institute of Technology, SE-100 44  Stockholm,
Sweden}
\email{henriksh@math.kth.se}

\thanks{H. A. was supported by postdoctoral fellowship from Knut and Alice Wallenberg Foundation. 
H. Sh. was partially supported by Swedish Research Council.}


\makeatletter
\def\@tocline#1#2#3#4#5#6#7{\relax
  \ifnum #1>\c@tocdepth %
  \else
    \par \addpenalty\@secpenalty\addvspace{#2}%
    \begingroup \hyphenpenalty\@M
    \@ifempty{#4}{%
      \@tempdima\csname r@tocindent\number#1\endcsname\relax
    }{%
      \@tempdima#4\relax
    }%
    \parindent\z@ \leftskip#3\relax \advance\leftskip\@tempdima\relax
    \rightskip\@pnumwidth plus4em \parfillskip-\@pnumwidth
    #5\leavevmode\hskip-\@tempdima
      \ifcase #1
       \or\or \hskip 2em \or \hskip 2em \else \hskip 3em \fi%
      #6\nobreak\relax
    \dotfill\hbox to\@pnumwidth{\@tocpagenum{#7}}\par
    \nobreak
    \endgroup
  \fi}
\makeatother

\begin{document}


\begin{abstract}    
The main purpose of the present paper is to establish a link
between quadrature surfaces (potential theoretic concept) and sandpile dynamics (Laplacian growth models).
For this aim, we introduce a new model of Laplacian growth on the lattice $\Z^d$ $(d\geq 2)$ which continuously deforms occupied regions
of the \emph{divisible sandpile} model of Levine and Peres \cite{Lev-Per10},
by redistributing the total mass of the system onto $\frac 1m$-sub-level sets of the odometer which is a function counting total emissions of mass
from lattice vertices. In free boundary terminology this goes in parallel with singular perturbation, 
which is known to converge to a Bernoulli type free boundary.

We prove that models, generated from a single source, have a scaling limit, 
if the threshold $m$ is fixed. Moreover, this limit is a ball, and the entire mass of the system is being 
redistributed onto an annular ring of thickness $\frac 1m$. 
By compactness argument we show that, when $m$ tends to infinity sufficiently slowly with respect to the scale
of the model, then in this case also there is scaling limit which is a ball,
with the mass of the system being uniformly distributed onto the boundary of that ball,
and hence we recover a quadrature surface in this case.

Depending on the speed of decay of $m$, the visited set of the sandpile interpolates between spherical
and polygonal shapes. Finding a precise characterisation of this shape-transition phenomenon seems
to be a considerable challenge, which we cannot address at this moment.
\end{abstract}

\maketitle

{\small{\tableofcontents}}

\section{Introduction}\label{sec-intro}

\subsection{Background}
In a recent work  \cite{AS}, the current authors introduced a new growth model on the lattice $\Z^d$
$(d\geq 2)$ which redistributes a given initial mass on $\Z^d$
onto a combinatorial free boundary. The growth rule
asks vertices of $\Z^d $ lying in the interior of the visited sites of the model,
and vertices on the boundary of the set of visited sites carrying mass larger than a
prescribed threshold, to redistribute their entire mass evenly among their $2d$ lattice neighbours.
This procedure  creates a sequence of non-decreasing
domains that was shown to converge to, what we termed,  {\it Boundary Sandpile} (hereinafter \textbf{BS} for short).  
This local rule allows (and in fact forces) huge masses to accumulate on the free boundary
of the moving front. 
In the case of a single source  mass, the numerics indicates that shapes generated by this model  do not converge to a sphere under a scaling limit, but to a shape somewhat reminiscent of the classical Abelian sandpile (see \cite{Bak} for definition, and \cite{PS} for the scaling limit). 

The initial motivation for BS  was to find a sandpile dynamic that gives us the so-called Quadrature surfaces  (QS) (see \cite{Sh94-1, Sh94-2}). 
 A QS, for a given source  $\mu$ 
  is the boundary of a domain $D$, that contains the support of $\mu$, and   has the property that the Poincar\'e balayage satisfies 
   \begin{equation}\label{balayage}
\mathrm{Bal}:  \mu   \longrightarrow    \cH^{d-1} \lfloor_{\partial D}  ,
\end{equation}
which is equivalent to 
 \begin{equation}\label{QS}
\int h(x) d\mu = \int_{\partial D} h(x) d \mathcal{H}^{d-1},
\end{equation}
for all $h$ harmonic on a neighbourhood of $\overline D$.  

Our intention in this paper is   to introduce a new sandpile dynamic that  (contrary to \textbf{BS})
will  correspond to a  QS in its scaling limit.  This would then parallel the theory of 
(well-behaved) divisible sandpiles (\textbf{DS}), which 
redistributes an initial mass by putting a prescribed amount of mass at each visited site,
and only moving out the excess from this prescribed amount.
In \textbf{DS} the limit shapes, even for multi-source masses (with some reasonable control over their distribution), have shown to be the so-called \emph{quadrature domains} (QD) with prescribed density and given source 
(see \cite{Lev-Per} for the single source, \cite{Lev-Per10} for multi-source case, also \cite{Lev-thesis} and
\cite{Lev-Per16}); a well-established area in potential theory \cite{GSh-2005}. 

To establish a link between QS and sandpile dynamics, we have chosen an approach based on singular perturbation theory for free boundary value problems (see \cite{BCN}). This approach suggests that  if one considers a slight perturbation of \textbf{DS} starting with total mass $n$ at the origin,
by putting larger mass $m$ on  $(\frac1m)$-sub-level sets of the odometer function, and letting $m=m_n$ tend to infinity as $n$ tends to infinity, 
then interesting shapes appear. By taking $m\approx n^{1/d}$ we obtain a shape close to that of \textbf{BS}, 
and by letting $m$ be fixed (but large) we obtain spherical shapes, see Figure \ref{Fig-many}.
The question that arose was: 
\smallskip

\begin{center}
{\it How fast/slow (relative to $n$) should  $m$  grow, in order to reach a desirable sandpile?}
\end{center}

\smallskip

More precisely, we are interested in finding various functions $F$, with  $m=F(n)$, for which there is
a new sandpile shape.  When $F(n) \asymp n^{1/d} $ it is obvious that we are close to \textbf{BS}, and when 
$F(n) =\mathrm{const}$, then we are having  a  \textbf{DS}\footnote{The only difference between DS and this case is that mass is now redistributed to the sub-level sets, with constant amount.}.

In this paper we show, using compactness arguments that there exists $m=F(n) \nearrow \infty $
for which the sandpile shape converge to a sphere, and the entire mass is being uniformly distributed onto the boundary.
 This in general suggests that in analogy with constructing QD through  \textbf{DS},  we can construct  QS through  sandpile dynamics too.

\begin{remark}{\normalfont(Technical remark for experts)}
In this paper, when studying the scaling limit of the model,  we only  consider  the case of a single source, and leave out the   general case. 
The reason for this is  several   infeasible  technical difficulties at the moment.
One major problem   for the case of general initial source is the uniqueness question for the QS. 
This may still not be an issue, if we can show that scaling limits of our problem  are  unique (as it was done in the case of Abelian Sandpile \cite{PS}, for instance). It is however far from 
obvious, even for a two-point source (see Figure \ref{Fig-2points} for a numerical illustration), how such a uniqueness can be proven.  In our case we have used the geometry of spheres to overcome this difficulty.
 
On the other hand it will be apparent (from what follows in this paper) 
that we can prove that for each  fixed $m$ subsequences of our model converge to solutions of  a free boundary problem  of the type
$$
\Delta u_m  = m \mathbb{I}_{\{0< u_m < 1/m\}} - \mu ,
$$
where $\mu$ is the initial source, consisting of several point masses.   A solution to this problem is 
not necessarily unique if $\mu$ is not a single point source. However, from free boundary technique 
we know that under certain good conditions (which is satisfied at least for a two-source point)
  the function $u_m$ converge (for a subsequence)
 to a function $\tilde u$, as $m$ tends to infinity, where $\tilde u$ solves a free boundary problem of Bernoulli type 
 $$
\Delta \tilde u  = c_0  \cH^{d-1} \lfloor_{\partial  \mathbb{I}_{ \{\tilde u > 0 \}  } } - \mu , \qquad (\hbox{some $c_0 >0$})
$$
which amounts to a quadrature identity of type \eqref{QS}. 

From here we want to deduce that by compactness we may consider a sequence  of $m=m_n$ such that $m_n \to \infty$, as $n \to \infty$, in a way that we can achieve the quadrature identity \eqref{QS} for the limit scenario.  For single source case the analysis heavily depends on the uniqueness and stability of the limit shape, that is unavailable for multi-source case.

\end{remark}

\begin{figure}[!htbp]
\centerline{%
\includegraphics[width=2in]{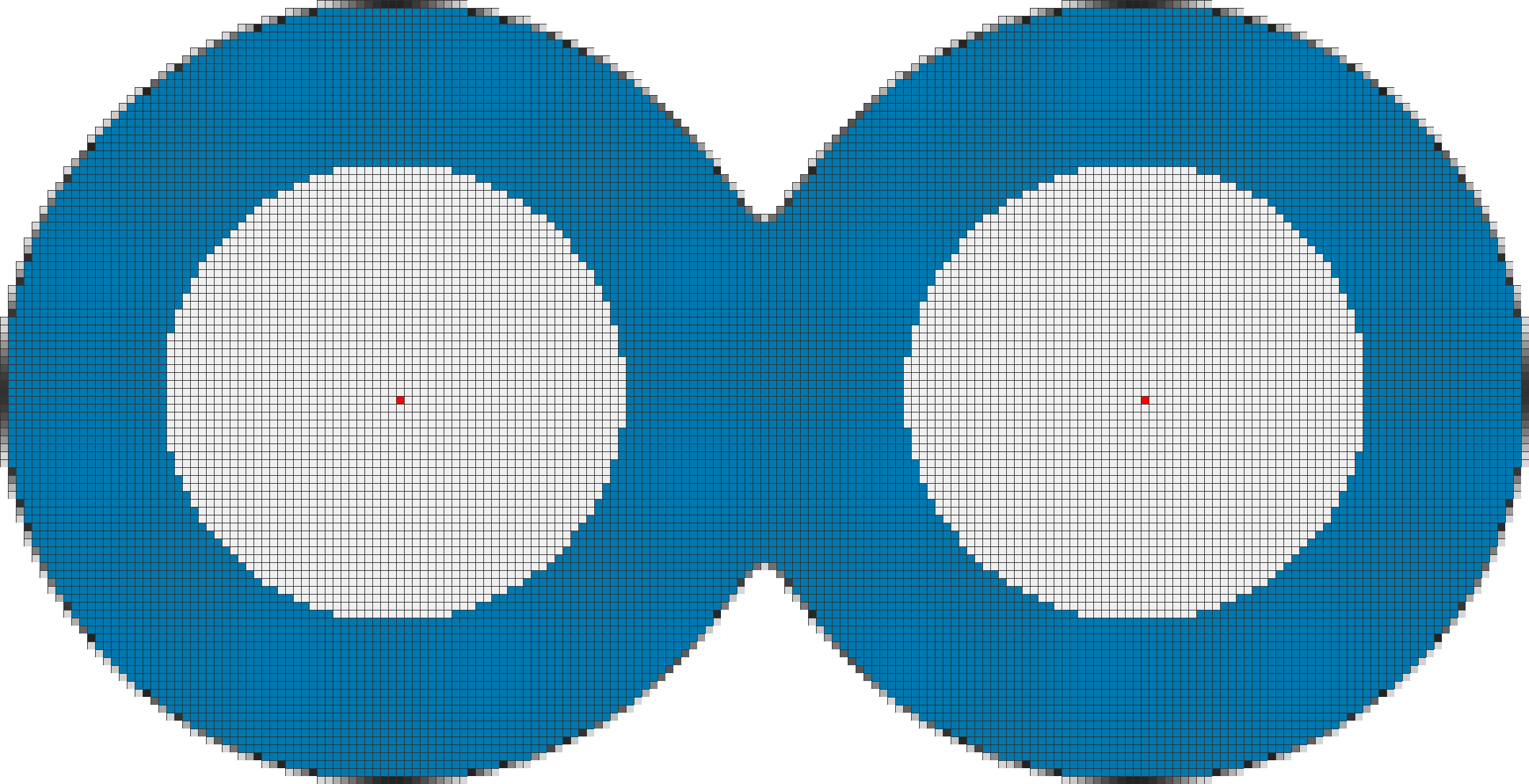}
\hspace{0.4cm}
\includegraphics[width=2in]{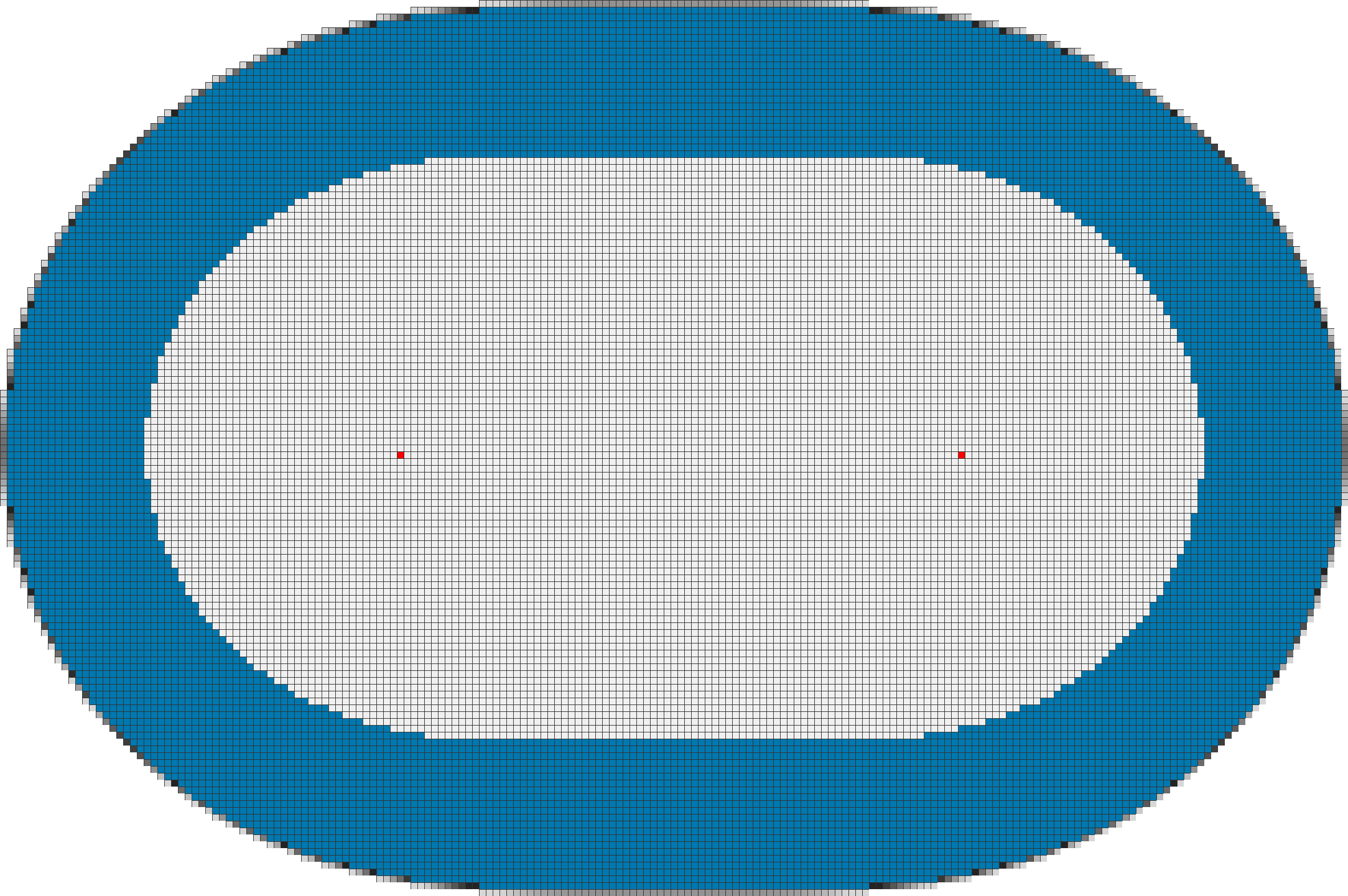}
}%
\caption{\footnotesize{Numerical simulation of the sandpile with initial distribution
concentrated at two points of $\Z^2$. On the left, each point $(\pm 47,0) \in \Z^2$
carries mass $50 \ 000$, and the sandpile threshold $m$ is set to $10$. On the right
we bring the two sources closer, by putting them at $(\pm 41,0)\in \Z^2$ and keeping the rest
of the parameters unchanged; the final configuration then evolves to a stadium shape.
The blue region in both images depicts points of $\Z^2$ having mass $m=10$,
and the gray area embraced by the blue region, carries no mass.
The two marked red points in the gray area represent the support of the initial mass distribution.}}
\label{Fig-2points}
\end{figure}

\begin{remark} {\normalfont (To non-expert readers)} Although it took a while to pull the right strings to get this result, we want to stress that the paper actually
is using soft mathematics, coming from classical theory of free boundary problems, translated into the discrete language. Hence we expect the reading to be easy.
We need also to stress that  the novelty of the paper lies in the model itself, and connections between free boundary problems and sandpile dynamics.
This we believe might have  some  potential  to grow in the near future and to give rise to  an independent research area.
There are indeed many free boundary problems on one side and several  particle dynamic processes on the other side that seem to be defining
the same model from different perspectives. Attempts to bridge these  models  should be valuable and interesting. 
\end{remark}

\subsection{Discrete Laplacian}\label{sub-sec-discrete-Lap}
Before getting into the technical definitions and details, 
we fix some notation, and several basic facts
which will be used throughout the text. 

Let $h>0$ be fixed, and $f: h\Z^d \to \R$ be a given function. Its (normalised) $\Delta^h$-Laplacian is defined as
$$
\Delta^h f(x) = \frac{1}{2d \hspace{0.02cm} h^2}  \sum\limits_{ y\sim_h x } [ f(y) - f(x) ],
$$
where $y\sim_h x$ means that $x,y\in h\Z^d$ are lattice neighbours, i.e. $||x-y||_{l^1} =h$.
When $h=1$, we will simply write $x\sim y$ instead of $x\sim_1 y$.
Similar to the discrete Laplace operator, one defines discrete derivatives.
Namely, for a unit vector $e\in \Z^d$, i.e. $|| e ||_{l^1}=1$, the function
$$
\nabla_e^h f(x) = h^{-1} \left[ f(x + h e ) - f(x) \right], 
$$
is the discrete derivative of $f$ in the direction $e$ at the point $x\in h\Z^d$.

It is well-known that as the lattice spacing tends to 0, the discrete Laplacian, and the discrete gradient converge to their continuous analogues, in a sense that
for any $\varphi \in C_0^\infty (\R^d)$ one has $2d \Delta^h  \varphi \to \Delta \varphi $, and $\nabla_e^h \varphi(x) \to \nabla_e \varphi(x)$ uniformly in $\R^d$, where $\Delta$ and $\nabla_e$ are respectively
the continuum Laplace operator and derivative in the direction $e$.

\smallskip
We next recall the definition of the \emph{fundamental solution} to $\Delta^1$ in $\Z^d$.
Namely, there exists a function $g(x,y):\Z^d \times \Z^d \to \R  $ such that
$\Delta^1_x g(x,y) = -\delta_0(y-x)$ for all $x,y \in \Z^d$, where $ \Delta^1_x $
is the Laplacian with respect to the $x$ variable, and $\delta_0$ is the characteristic function of $0\in \Z^d$.
The existence and asymptotics of such $g$ are well-known (see \cite{Fuk-Uc}, \cite{Uchiyama}, \cite{Lawler-book-walks})
and we have
\begin{equation}\label{g-x-y}
 g(x,y) =   \begin{cases}  -\frac{2}{\pi} \log|x-y|  + \gamma_0 + \mathrm{O}(|x-y|^{-2}) , &\text{$d=2$}, \\ 
  \frac{2}{(d-2) |B_1|} |x-y|^{2-d} + \mathrm{O}(|x-y|^{-d})  , &\text{$d \geq 3$},    \end{cases}
\end{equation}
where $\gamma_0$ is a constant, and $|B_1|$ is the volume of the unit ball in $\R^d$.

\smallskip
The operator $\Delta^1$ also enjoys a maximum principle
(see for instance \cite[Exercise 1.4.7]{Lawler-book-walks}).

\vspace{0.2cm}
\noindent \textbf{Discrete maximum principle (DMP): } Let $u,v: \Z^d \to \R$, $V\subset \Z^d$ be finite,
 and $\Delta^1 u \geq \Delta^1 v$ in the interior of $V$.
Then $\max_{\partial V} (u-v) \geq \max_{V} (u-v)$.

\smallskip
In this statement, and throughout the paper, the boundary of a non-empty set $E\subset \Z^d$,
denoted by $\partial E$,
is the set of all $x\in \Z^d$ for which there exists $x\in E$ such that $x\sim y$.

\subsection{Our model and the main results}\label{sub-sec-model}
We now give the formal definition of the sandpile process which is the focus of this paper.
Throughout the paper we will always assume that the dimension $d$ of the space 
is at least two.
Assume we have an initial distribution of mass $\mu_0 :\Z^d \to \R_+$ $(d\geq 2)$, that is a non-negative bounded function of finite support on $\Z^d$, and
let $n\geq 0$ be the total mass of the system, i.e. $n=\sum_{x\in \Z^d} \mu_0(x)$. Fix also some threshold $m>0$,
and inductively define a triple $(V_k, u_k, \mu_k)$ for integer $k\geq 0$ as follows. When $k=0$ set $V_0 = \mathrm{supp} \mu_0$
and $u_0 \equiv 0$; we now define the passage from $k$ to $k+1$.
At a time $k$ call a particular site $x\in \Z^d$ \emph{unstable} if either of the following two (mutually exclusive) conditions
hold:
\begin{itemize}
 \item[(a)] $\mu_k(x)>m $,
\smallskip 
 \item[(b)] $0<\mu_k(x)\leq m  $ and $u_k(x) > \frac 1m {n^{ \frac 2d }}$.
\end{itemize}

For an unstable site $x$ we define an \emph{excess} of mass denoted by $\mathcal{E}(x)$
and being equal to $\mu_k(x) - m  $ if $x$ is unstable according to (a),
and to $\mu_k(x)$, if (b) holds for $x$.
Then, a \emph{toppling} of an unstable site $x $ is the procedure
of redistribution of the excess $\mathcal{E}(x)$ evenly among its $2d$ lattice neighbours.

Thus, at each discrete time $k\geq 0$ we identify an unstable site, say $x_k$ and topple it. 
Then, we set $V_{k+1} : = V_k \cup \{ y\in \Z^d: \ y\sim x \} $, 
$u_{k+1}(x) = u_k (x)  + \delta_x \mathcal{E}(x) $, $x\in \Z^d$, and 
\begin{equation}\label{mu-k+1-def-A}
 \mu_{k+1} (y) =  \begin{cases}  \mu_k(x) - \mathcal{E}(x), &\text{if $y=x$}, \\ 
 \mu_k(y) + \frac{1}{2d} \mathcal{E}(x)   ,&\text{if  $ y\sim x    $}, \\
 \mu_k (y), &\text{otherwise}.    \end{cases}
\end{equation}
The process terminates if there are no unstable sites.
It follows easily from the definition of the process, that for any $k\in \Z_+$ we have
\begin{equation}\label{Laplace-1}
 \Delta^1 u_k(x) = \mu_k (x ) - \mu_0(x), \qquad x\in \Z^d.
\end{equation}

\bigskip

Let $T=\{x_k\}_{k=0}^\infty \subset \Z^d$ be any infinite toppling sequence, and $u_k$ and $\mu_k$ be the odometer, and mass distribution
at time $k$ as defined above. Obviously, for each fixed $x\in \Z^d$ the sequence $\{u_k(x)\}$ is non-decreasing,
and hence there is a limit (odometer) $u (x) : = \lim_{k\to \infty} u_k(x)$, possibly infinite.
Moreover, if $u $ is finite at a given $x$ and its lattice neighbours, then passing to the limit in \eqref{Laplace-1}
we see that $\mu_k(x)$ is also convergent; let $\mu(x)$ be its limit in such case.
With this notation, we fix the following:

\begin{definition}\label{defn-stable-seq}{\normalfont (Stabilizing sequence)}
Call a sequence $T$ stabilizing, if the limit odometer $u$ is finite everywhere on $\Z^d$, and the final configuration $(u, \mu)$
is stable.
\end{definition}

For convenience, we will allow a toppling sequence $T$ to include stable sites as well. Thus, if a certain site $x$ included in $T$
is assigned to topple at time $k$, but is stable, then toppling $x$ does not change anything.
In other words the toppling procedure for stable vertices is idle, and the triple $(V_k, u_k, \mu_k)$
simply remains the same for time $k+1$.

\bigskip

\begin{figure}[!htbp]
\centerline{%
\includegraphics[width=1.5in]{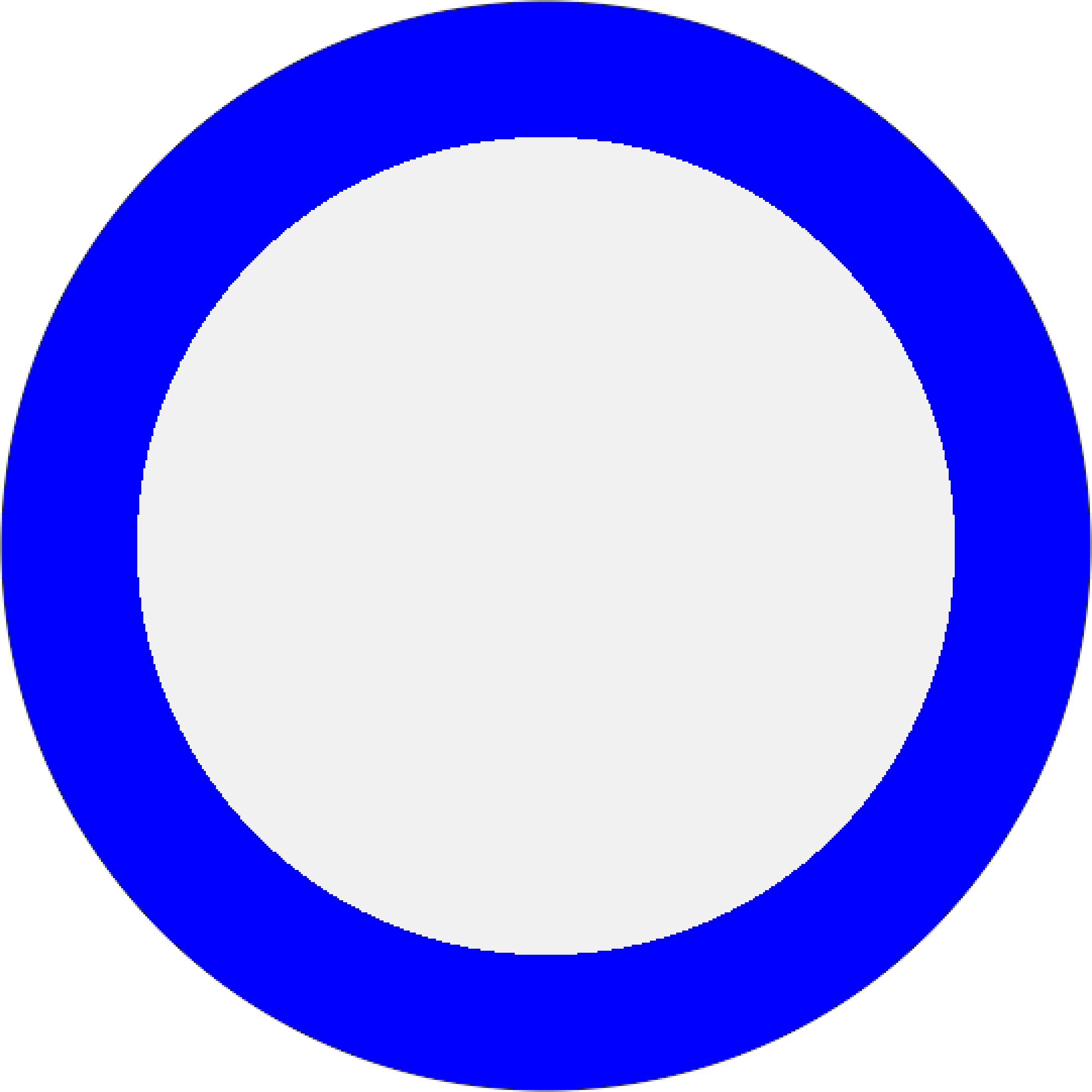}
\hspace{0.2cm}
\includegraphics[width=1.5in]{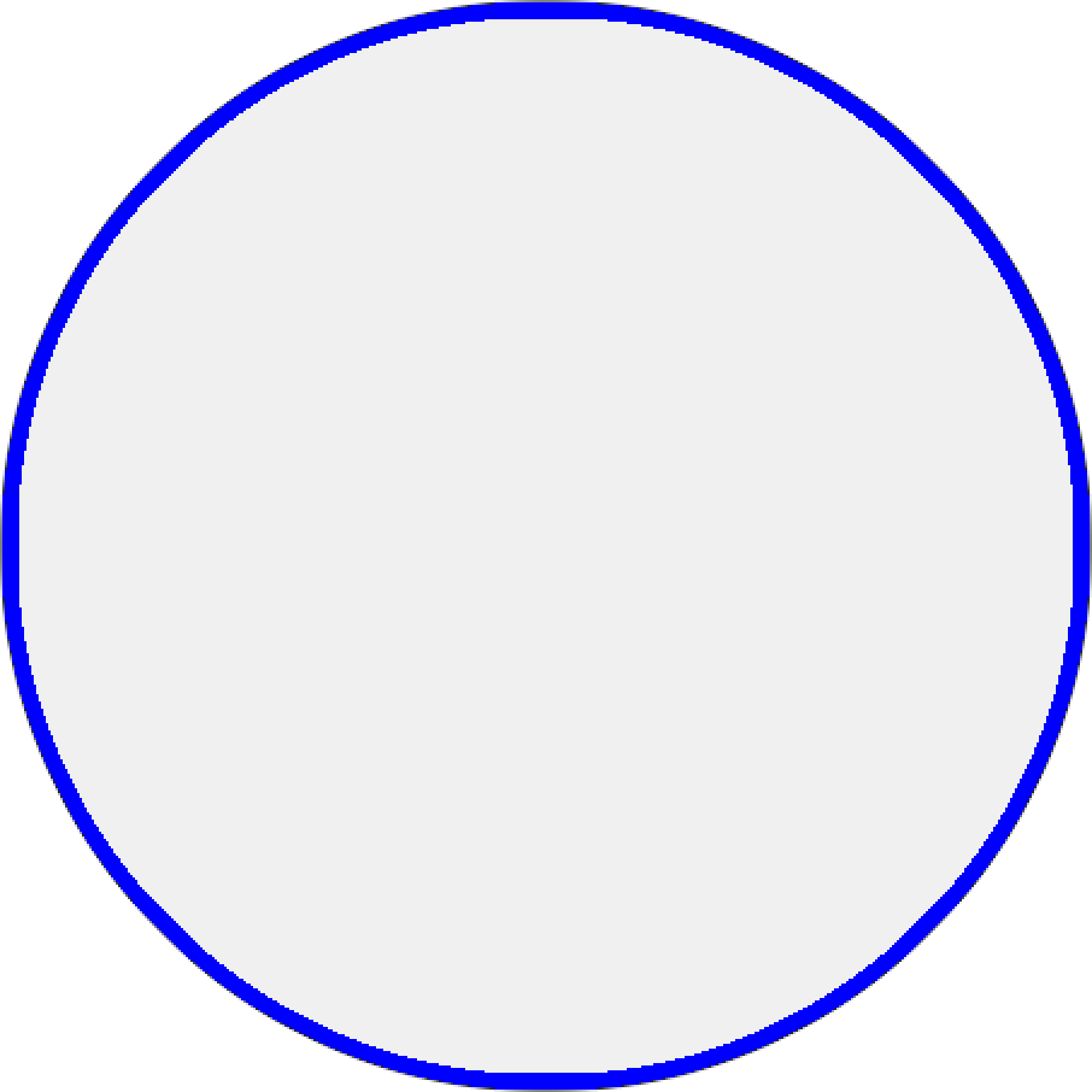}
}
\vspace{0.2cm}
\centerline{%
\includegraphics[width=1.5in]{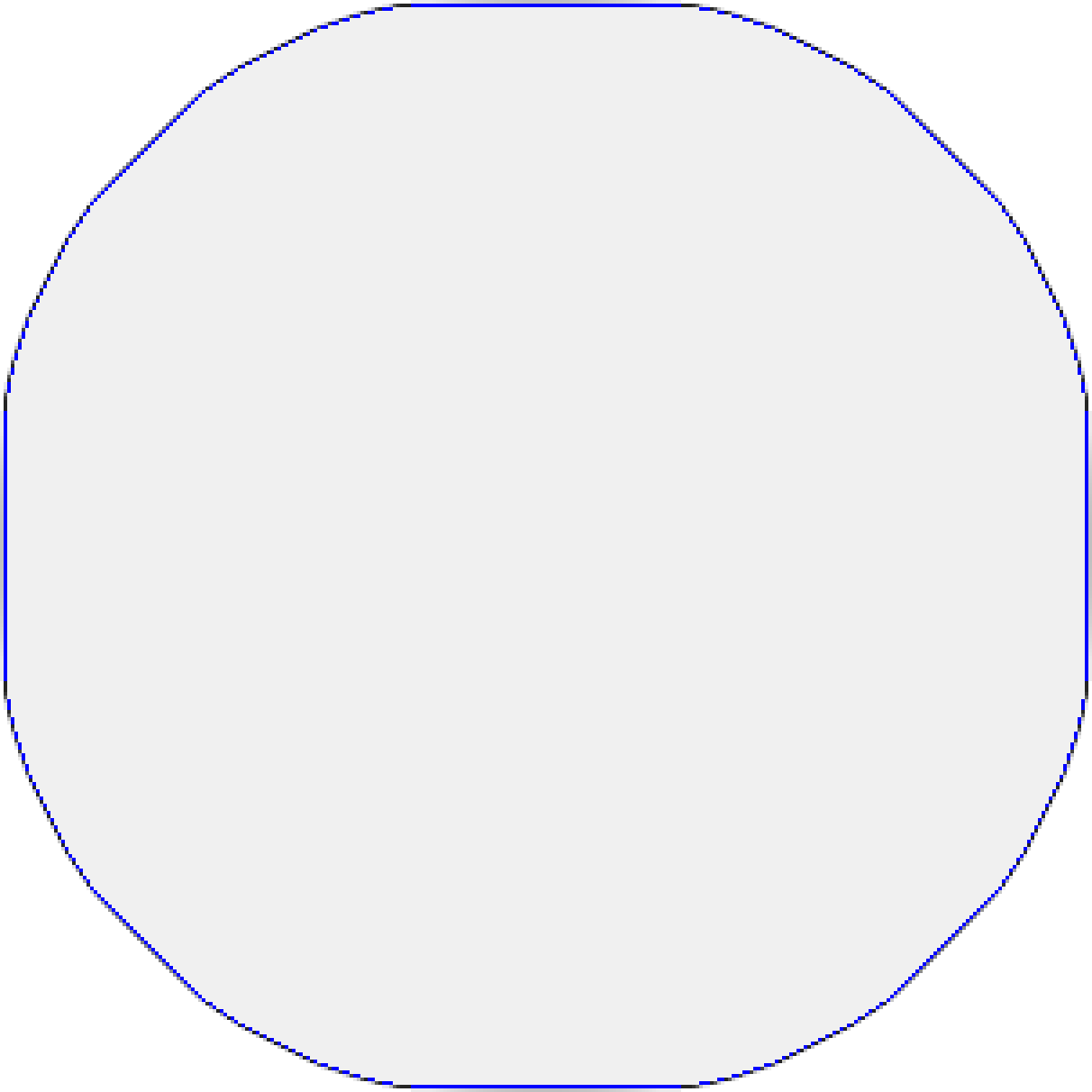}
\hspace{0.2cm}
\includegraphics[width=1.5in]{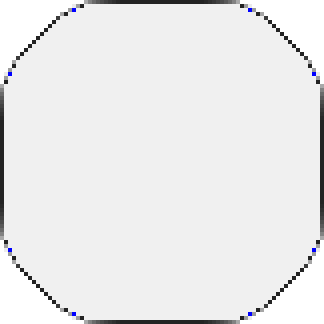}
}%
\caption{\footnotesize{From top left to bottom right are the 
final configurations of the sandpiles on $\Z^2$ starting from mass $10^6$ at the origin,
and with threshold $m$ equal to $10$, $100$, $1 \ 000$ and $5 \ 000$
respectively. The colouring scheme is the same as in Figure \ref{Fig-2points},
and the size of the sandpile regions are scaled to have the same resolution as an image.
In reality, the regions in this figure for larger values of the threshold $m$, 
have smaller size in $\Z^2$.
On the boundary, darker colors indicate higher concentration of mass.
One observes, a change from a robust spherical shape to a polygonal shape of the \textbf{BS}
as the threshold increases.}}
\label{Fig-many}
\end{figure}

\noindent \textbf{Main results and organisation of the paper.} 
Our main results concern general analysis in the
discrete space $\Z^d$ of the singularly perturbed sandpile model
defined in subsection \ref{sub-sec-model}, and the existence of the scaling limit of the model
generated from a single source. More precisely, in Section \ref{sec-basic} we show
that the model is well-posed, in a sense that it reaches a stable configuration given that all unstable sites
have been toppled infinitely many times (Proposition \ref{prop-terminate}), moreover, this stable configuration 
does not depend on the order of topplings (Proposition \ref{prop-abelian}).
We then show that the odometer function is the smallest element in the class of all
super-solutions to discrete PDE problem associated with the sandpile (Lemma \ref{Lem-odometer-is-u-star}).
Specifying the analysis for initial distributions concentrated at a single point of $\Z^d$,
we show that the odometer function is monotone in the lattice directions (Theorem \ref{Thm-monotonicity}).
We then estimate the size of the annulus where the mass is concentrated (Lemma \ref{Lem-size-of-mass-region}), and
of the region which is free of mass (Proposition \ref{prop-size-of-mass-FREE-region}) (cf. Figure \ref{Fig-many}). 
Finally, we complete Section \ref{sec-single-source} by proving that the odometer
is Lipschitz regular uniformly with respect to threshold $m$ (Proposition \ref{prop-Lip}) and is $C^{1,1}$
with norm depending linearly on $m$ (Proposition \ref{prop-1-step-C11-log}).

In Section \ref{sec-scaling-lim} we show that the sandpile shapes generated
by an initial distribution of the form $n\delta_0$ and fixed threshold $m$,
converge to a ball under a scaling limit, as $n\to \infty$, 
moreover, the limiting odometer solves a 
singular perturbation problem with Bernoulli type free boundary condition (Theorem \ref{Thm-scaling-limit-fixed-M}).
We also prove, that there is a scaling limit for a subsequence of the sandpiles, as the
mass threshold $m$ tends to infinity slowly with $n$ (Theorem \ref{Thm-scaling-limit-inf-M}).

The appendix to the paper, proves a uniqueness result for solutions to singular perturbation problem with Bernoulli type free boundary condition. This result is being used for establishing the existence of the scaling limit of the model.


\section{Basic properties of the model}\label{sec-basic}

\subsection{Termination and abelian property}

The aim of this section is to prove the well-posedness of the model (in a sense to be made precise below),
and to present some basic properties of it.

It is not hard to realize that the process defined in subsection \ref{sub-sec-model}
will not stabilise after finitely many steps (modulo the trivial cases).
Moreover, vertices in the set of visited sites may become unstable infinitely many times
during the lifetime of the process. This, as in the case of BS \cite{AS}, brings us to he following class of the toppling
sequences, which leaves a chance for the process to reach a stable state.

\begin{definition}\label{defn-infint-seq}{\normalfont (Infinitive sequence)}
Call a sequence $T=\{x_k\}_{k=1}^\infty \subset \Z^d$ infinitive (on a set $E\subset \Z^d$)
if any $x\in \Z^d$ (correspondingly $x\in E$) appears in $T$ infinitely often.
\end{definition} 

If not stated otherwise, in this section $\mu_0$ is any initial distribution of mass,
with total mass $n >0$, and $m>0$ stands for the threshold of the sandpile process.
Note, that $n$ and $m$ are not necessarily integers.

The next two propositions show that any initial distribution can be stabilized uniquely by an infinitive toppling sequence.
Proofs follow closely to their corresponding results from our paper \cite{AS},
without any additional difficulties, but we include some details here for the sake of completeness.

\begin{prop}\label{prop-terminate}{\normalfont{(Termination of the process)}}
If a toppling sequence  $T$ is infinitive, then it is stabilizing according to Definition \ref{defn-stable-seq}.
\end{prop}

\begin{proof}
As in the proof of \cite[Proposition 2.1]{AS} it will be enough to show that the set of visited
sites $V_k$ stays bounded uniformly in $k$, for which it suffices to obtain a uniform bound on the number of elements in $\partial V_k$
(see the scanning procedure by hyperplanes as described in \cite[Proposition 2.1]{AS}).

Let $k\geq 1$ be given, and fix any $x\in \partial V_k \setminus V_0 $, clearly $u_k(x) =0$.
Assume $y\sim x$ is the generator of $x$, i.e. the first time $x$ was visited by the process, was due to the toppling of $y$,
say at time $i<k$. Since $x\sim y $  is not visited at time $i$, it follows that $y\in \partial V_i$, in particular
$u_{i-1}(y)=0$, and hence $y$ can topple only according to condition (a) for instability as defined in subsection \ref{sub-sec-model},
and hence $y$ must carry mass at least $m$.
This means that at time $k$, when $x\in \partial V_k$, the 1-neighbourhood of $x$ has total mass at least $m/(2d)$.

For each $x\in \Z^d$ let $\mathrm{B}(x) $ be a 1-box centred at $x$, i.e. $\mathrm{B}(x) = \{x\}\cup \{y\in \Z^d: \ y \sim x\}$.
On one hand we proved that the total mass carried
by vertices of $\mathrm{B}(x) $ where $x\in \partial V_k \setminus V_0$ is  at least $m/(2d)$.
On the other hand, each $x\in \partial V_k \setminus V_0$ can be counted at most $2d$ times in each box $\mathrm{B}$,
and hence, we conclude that the total number of boundary points is bounded above by $ (2d)^2 \frac{n}{ m } $ .
Since the bound is uniform in $k$, the proof of this proposition is complete.
\end{proof}

\begin{prop}\label{prop-abelian}{\normalfont{(Abelian property)}}
For any two infinitive toppling sequences $T_1$ and $T_2$ the corresponding final configurations are the same. 
\end{prop}
\begin{proof}
The scheme of the proof follows that of \cite[Proposition 2.2]{AS}.
In view of the previous proposition, each $T_i$
is stabilizing and has well-defined odometer function, call it $u_i$.
We will prove that $u_2(x) \geq u_1(x)$ for any
$x\in \Z^d$, which by symmetry will imply the desired result.

For $i=1,2$ by $u_{i,k}$ denote the odometer function and by $\mu_{i,k}$ the distribution 
corresponding to the sequence $T_i$ after the $k$-th toppling occurs.
Let $T_1= \{ x_k \}_{k=1}^\infty$,
we now show that 
\begin{equation}\label{ineq-abelian}
u_2(x_k ) \geq u_{1,k} (x_k ), \qquad k=1,2,... \ . 
\end{equation}
Observe that (\ref{ineq-abelian}) is enough for our purpose
since the site $x_k $ appears infinitely often in $T_1$ and $u_{1,k} $ converges to $u_1$ pointwise as
$k$ tends to infinity. 
Thus, in what follows we prove (\ref{ineq-abelian}) which we will do by induction on $k$.

The base case of induction, i.e. when $k=1$, is trivial. We now assume that
(\ref{ineq-abelian}) holds for any time $1\leq i<k$, and prove it for time $k$.
First we eliminate the trivial case, when $x_k $ in $T_1$ is stable at time $k$.
Indeed, in such scenario, we get
\begin{equation}\label{non-legal}
u_{1,k}( x_k ) = u_{1,k-1}( x_k ).
\end{equation}
So, if $x_k $ has never toppled prior to time $k$, then $u_{1,k-1}( x_k )=0$
and we are done. Otherwise, if $i\leq k-1$ is the last time $x_k$
has toppled in $T_1$, then clearly $x_i =x_k $,
and hence, using the inductive hypothesis, we conclude
$$
u_2(x_k ) =u_2(x_i ) \geq u_{1,i}(x_i ) = u_{1,k-1}(x_i ) =  u_{1,k}(x_k),
$$
where the last equality follows from (\ref{non-legal}).
This completes the induction step for the case when $x_k$ was already stable.

We next proceed to the case when toppling $x_k$ in $T_1$ is not stable at time $k$.
First we will prove that for any $x\neq x_k$  one has
\begin{equation}\label{u2-geq-u1k}
u_2(x)\geq u_{1,k} (x).
\end{equation}
There are two possible sub-cases, either $x$ was never toppled in $T_1$ up to time $k$,
which implies $u_{1,k} (x)=0$ and we get (\ref{u2-geq-u1k}),
or  $x$ was toppled at some time before $k$.
In the latter case let $i\leq k$ be the last time $x$ has toppled prior to time $k$.
Observe that $i<k$ since $x\neq x_k $. Then we have
$$
u_{1,k} (x) = u_{1,i}(x) \leq u_2(x),
$$
where the second inequality follows by inductive hypothesis.
We thus have proved (\ref{u2-geq-u1k}) for all $x\neq x_k $.
Consider the following inequality
\begin{equation}\label{mu2-dist-vs-mu1k}
\mu_2( x_k ) \leq \mu_{1,k} (x_k ).
\end{equation}
To complete the induction step suppose for a moment that (\ref{mu2-dist-vs-mu1k}) holds
true. Then
\begin{multline*}
\frac{1}{2d} \sum\limits_{y \sim x_k  } u_2(y)  - u_2(x_k ) = 
  \Delta^1 u_2( x_k ) =\mu_2(x_k ) - \mu_0(x_k ) \stackrel{ (\ref{mu2-dist-vs-mu1k}) }{\leq} \\ 
  \mu_{1,k} (x_k) - \mu_0(x_k ) = 
  \Delta^1 u_{1,k}(x_k) =\frac{1}{2d} \sum\limits_{y \sim x_k } u_{1,k}(y)  - u_{1,k}( x_k ).
\end{multline*}
Rearranging the first and the last terms leads to
$$
u_2( x_k ) - u_{1,k}( x_k ) \geq \frac{1}{2d}\sum\limits_{y \sim x_k } \left( u_2(y) - u_{1,k} (y) \right) \geq 0,
$$
where the last inequality is due to (\ref{u2-geq-u1k}). This completes the induction.
Thereby, to finish the proof of the proposition, we need to verify (\ref{mu2-dist-vs-mu1k})
which we do next.

Recall that $x_k $ is unstable, and hence there are two possible reasons for $x_k$ to topple in $T_1$ at time $k$. Namely,
\begin{itemize}
 \item[(a)] $\mu_{k-1,1} (x_k ) >m $ 
 \item[(b)] $0< \mu_{k-1,1} (x_k ) \leq  m  $ and $u_{1,k-1}(  x_k ) >\frac 1m n^{\frac 2d}  $.
\end{itemize}
If we are in case (a), then $\mu_{1,k}( x_k ) = m$ in view of the toppling rule, and hence (\ref{mu2-dist-vs-mu1k})
follows in view of the stability of $\mu_2$.
Next, if $x_k $ is set to topple because of (b), then $u_{1,k-1}( x_k  ) > \frac 1m n^{\frac 2d}$.
But this means that
$x_k$ had already toppled prior to time $k$, and we let $i<k$ be the largest time before $k$, when it has toppled;
in particular $x_k = x_i$.
Applying the inductive hypothesis we arrive at
$$
u_2(x_k ) = u_2(x_i  ) \geq   u_{1,i}( x_i ) = u_{1,i}( x_k ) >\frac 1m n^{\frac 2d}.
$$
From here, and the stability of $\mu_2$ we obtain $\mu_2 (x_k )=0 $, completing the proof of
\ref{mu2-dist-vs-mu1k}), and hence the proposition is proved.

\end{proof}

A simple consequence of the abelian property, is the following
symmetry for a point mass concentrated at the origin.
Consider the set of vectors 
$$
\mathcal{N} : = \{e_i, \ e_i \pm e_j: \ 1\leq i\neq j \leq d \},
$$
where $e_i\in \Z^d$ is the $i$-th element of the standard basis. 
Let $T$ be a hyperplane through the origin of $\Z^d$ and with a normal collinear with some element of $\mathcal{N}$.
Then, any sandpile generated by initial distribution of the form $n\delta_0$, is symmetric with respect to $T$. Namely, if $u$ is the odometer,
and $x\in \Z^d$ is arbitrary,  for $x^* \ -$the mirror reflection of $x$ with respect to $T$ (which is obviously in $\Z^d$ in view of the choice of $T$)
one has 
\begin{equation}\label{od-symm}
u(x)  = u(x_*). 
\end{equation}

For coordinate directions, i.e. when $T$ has normal in the direction of some $e_i$, the proof follows
by symmetrization of the toppling sequence (see \cite[Corollary 2.4]{AS}).
In case of the directions $e_i \pm e_j$, the claim follows by noticing that the symmetry with respect
to $T$ is a composition of two reflections with respect to coordinate axes.

\bigskip

The next result, which is yet another easy corollary of the abelian property,
will be used in the proof of Lemma \ref{Lem-odometer-is-u-star}.

\begin{lem}\label{Lem-Abelian-barrier}
Assume we are given a function $u_b:\Z^d\to \R_+$ and an
infinitive toppling sequence $T=\{x_k\}_{k=0}^\infty$.
Suppose in addition, that at any time $k\geq 0$ we require the total (cumulative) emissions
of mass from $x_k$ to NOT exceed $u_b(x_k)$.
Under this restriction, if the limiting configuration is stable, 
then it coincides with the limiting (stable) configuration of the sandpile.
\end{lem}

\begin{proof}
For each $k\in \Z_+$ let $u_k$ be the odometer function under the additional restriction
imposed by the barrier $u_b$, and let $\mu_k$ be the mass distribution at time $k$.
We will show, that for each $k$ toppling $x_k$ makes the site $x_k$
stable, meaning that the requirement $u_k(x) \leq u_b(x)$ for all $k\in \Z_+$ and all $x\in \Z^d$,
can be dropped. Then, the claim of the lemma will follow by the abelian
property proved in Proposition \ref{prop-abelian}.

Fix $k\geq 0$, then if $x_k$ is stable,
no toppling will be performed at time $k$, otherwise, there are two mutually exclusive reasons
for instability of $x_k$ at time $k$. Namely, 
\begin{itemize}
 \item[(a)] $\mu_{k} (x_k ) > m $ 
 \smallskip
 \item[(b)] $0< \mu_k (x_k ) \leq  m  $ and $u_k(  x_k ) > \frac 1m n^{\frac 2d} $.
\end{itemize}

Assume we are in case (a). Then, according to the sandpile rule defined in subsection \ref{sub-sec-model}
we need to topple the excess from $m$, i.e. $\mu_k (x_k ) - m$. Following the restriction
imposed by the barrier $u_b$, the maximum amount of mass we can move out from $x_k$
equals $u_b(x_k) - u_k(x_k) $. Now assume that 
$$
u_b(x_k ) <  \mu_k(x_k) - m + u_k ( x_k ) ,
$$
i.e. we cannot move the entire excess. Set 
$$
\e_0 : =  \mu_k (x_k) - m + u_k ( x_k ) - u_b(x_k )  >0.
$$
Then, invoking the $k$-th toppling, we will move $u_b(x_k) - u_k(x_k)$, meaning that $u_{k+1}(x_k ) = u_b(x_k)$.
Hence,
$$
\mu_{k+1}(x_k) - m = \mu_k(x_k) - m + u_k(x_k) - u_b(x_k) =\e_0>0,
$$
and since no further toppling of $x_k$ will be allowed, the last expression shows that
that $x_k$ remains unstable in the limiting configuration, which is a contradiction.
We thus obtain, that in case (a) the usual toppling rule is not affected by the existence of a barrier $u_b$.

The same reasoning gives the claim in case (b) too, and hence the proof of the lemma is completed.
\end{proof}

\bigskip

Having a unique stable configuration for the sandpile, 
for fixed initial distribution $\mu_0 $ we set $u:\Z^d \to \R_+$, $V\subset \Z^d$,
and $\mu:\Z^d\to \R_+$ for its odometer, set of visited sites, and the final distribution
respectively. In particular we have 
$$V= \supp\mu_0 \cup \{u>0\} \cup \partial \{u>0\} .$$
Let us also see what is the discrete PDE problem solved by odometer.
If $x\in \Z^d$ and $0<u(x)\leq \frac 1m n^{2/d}$, then $x$ has always toppled according to rule (a),
and hence $\mu (x)  = m$. Next, if $u(x)> \frac 1m n^{2/d}$, 
then eventually $x$ starts
to topple according to (b), and hence we get $\mu(x) = 0$.
Finally, if $x\in \partial \{u>0\}$ then $u(x) =0$, and hence $x$ has never toppled, which implies
$\mu \leq m$. Summarising the discussion, we conclude
\begin{equation}\label{eq-odometer1}
\Delta^1 u(x) + \mu_0(x) \leq m, \qquad x\in \Z^d,              
\end{equation}
\begin{equation}\label{eq-odometer2}
\Delta^1 u(x) + \mu_0(x)  = m \mathbb{I}_{ \left\{ 0<u\leq \frac 1m n^{ 2/d}  \right\} } (x) ,  \qquad x\in \{u>0\},               
\end{equation}
where $\mathbb{I}$ stands for the characteristic function of a set.
Observe, that if the initial distribution $\mu_0$ is already stable, then the odometer is identically zero.

\subsection{The smallest super-solution}
We define a class of super-solutions to our sandpile model.
\begin{definition}\label{def-super-sol}\normalfont{(Super-solutions)}
Let $\mu_0:\Z^d \to \R_+$ be a given initial distribution  with total mass $n$, and let $m > 0$ be fixed.
Call $u:\Z^d \to \R_+$ a super-solution to the sandpile if
\begin{itemize}
 \item[(i)]  $\Delta^1 u(x) +\mu_0(x)    \leq  m \qquad x\in \Z^d$,
 \vspace{0.2cm}
 \item[(ii)] $\Delta^1 u(x) + \mu_0(x)   \leq  m \mathbb{I}_{ \left \{ 0<u\leq \frac 1m n^{2/d}  \right\} } (x)  \qquad x\in \{ u>0 \}$,
\end{itemize}
\end{definition}

Denote by $\mathbb{W}$ the set of super-solutions, which is non-empty since the odometer of
the sandpile is in $\mathbb{W}$. We have the following.

\begin{lem}\label{Lem-smallest-super}
If $u_*(x) = \inf\limits_{w\in \mathbb{W}} w(x)$, $x\in \Z^d$, then $u_* \in \mathbb{W}$.
\end{lem}
\begin{proof}
Let us first show that for any $w_1, w_2 \in \mathbb{W}$ one has $u: = \min\{w_1, w_2 \}\in \mathbb{W}$.
Fix $x\in \Z^d$ and assume $0<u(x)\leq \frac 1m n^{2/d}$. Then, $w_i(x) >0$ for $i=1,2$,
and if $u(x) = w_1(x) $ then $w_1(x)\leq \frac 1m n^{2/d}$.
From this we get
$$
\Delta^1 u(x) + \mu_0(x) \leq \Delta^1 w_1(x) +\mu_0(x) \leq m,
$$
where the second inequality is in view of $w_1\in \mathbb{W}$.
Similarly, if $u(x)>\frac 1m n^{2/d}$, then both $w_1(x), w_2(x)> \frac 1m n^{2/d}$, and arguing as above
we obtain $\Delta^1 u(x) + \mu_0(x) \leq 0$.
We thus have condition (ii) for $u$, and (i) follows similarly, in a straightforward manner.
In particular, we get that $\mathbb{W}$ is closed under taking minimum of finitely many elements.

We next show that $u_* \in \mathbb{W}$. Let $u_0$ be the odometer of the sandpile.
Since $u_0\in \mathbb{W}$, and hence $0\leq u_* \leq u_0$, we get that $u_*$ has finite support,
and we let $E=\{x_1,...,x_N\}$ be the closure of the support of $u_*$, i.e. the support of $u_*$ along with its lattice boundary.
Then, for any $\e>0$ and each $1\leq i\leq N$ there exists $w_\e^i \in \mathbb{W}$
such that $0\leq w_\e^i(x_i) -   u_*(x_i) \leq \e $. Setting
$w_\e:=\min\{w_\e^1,...,w_\e^N, u_0\}$, we obtain $w_\e\in\mathbb{W}$ and
\begin{equation}\label{w-e-approx}
0\leq w_\e (x) - u_*(x) \leq \e \  \text{ for all } x\in \Z^d.
\end{equation} 
Since $\e>0$ is arbitrary, the claim of the lemma follows from \eqref{w-e-approx} by
a similar argument as we had above.
\end{proof}

\begin{lem}\label{Lem-equality-u-star}
Let $u_*$ be the smallest super-solution defined in Lemma \ref{Lem-smallest-super}.
Then, $u_*$ satisfies Definition \ref{def-super-sol} with equality in (ii).
\end{lem}
\begin{proof}
Let $V_*=\{u_*>0\}$ which is a finite set in $\Z^d$. If $V_*$ is empty we are done,
otherwise, set $E=\{x: \ 0<u_*(x) \leq \frac 1m n^{2/d} \} \subset V_*$, and let $u$ be the solution to
\begin{equation}\label{u-def-star}
\Delta^1 u(x) + \mu_0(x) = m \mathbb{I}_E(x) \text{ in } V_* \qquad \text{and} \qquad 
u=0 \text{ on } \partial V_*.
\end{equation}
It follows from the definition of $u$ and $u_*$ that $u-u_*$ is a sub-solution,
hence from DMP we get that $u\leq u_*$ in $V_*$.
Since both $u$ and $u_*$ are 0 outside $V_*$ we obtain $u\leq u_*$
on $\Z^d$. Let us now prove that $u$ is a super-solution;
this will imply that $u=u_*$ and then \eqref{u-def-star} will complete the proof.

We start with part (i) of Definition \ref{def-super-sol}.
If $x\in V_*$ then inequality of (i) follows from \eqref{u-def-star},
otherwise, if $x \in \Z^d \setminus V_*$ we have $u(x)=u_*(x)=0$ which, combined with (i)
for $u_*$, and the inequality $u\leq u_*$ on $\Z^d$, implies condition (i) of a super-solution for $u$.
We next proceed to (ii). It is clear from \eqref{u-def-star} that it suffices 
to show that $\Delta^1 u(x) + \mu_0(x) \leq 0$ if $u(x)>\frac 1m n^{2/d}$.
But since $u_* \geq u $ everywhere, for such $x$ we have $u_*(x)>\frac 1m n^{2/d}$ and
hence $x\notin E$. The latter and \eqref{u-def-star} imply  $ \Delta^1 u(x) + \mu_0(x) =0$,
and the claim of the lemma follows.
\end{proof}

\begin{lem}\label{Lem-odometer-is-u-star}
Let $u_*$ be as in Lemma  \ref{Lem-smallest-super}, and let $u_0$ be the odometer of the sandpile.
Then $u_* = u_0$.
\end{lem}

\begin{proof}
If $u_*=0$, then the sandpile is stable, and we are done, since then $u_0=0$ as well.
Now assume that $V_*:=  \{u_*>0\}$ is not empty. Fix any infinitive toppling sequence on $V_*$,
 call it $T=\{x_k\}_{k=0}^\infty$,
and do the toppling following $T$, with an additional requirement that 
at any time $k$ the cumulative emissions of mass from any $x\in \Z^d$ cannot
exceed $u_*(x)$.
Let the limiting odometer for this modified process be $v_0$. We claim that $v_0 = u_*$.
Observe, that this equality, combined with Lemma \ref{Lem-Abelian-barrier} and stability of $u_*$ completes the proof of the current lemma.
In what follows we prove the desired equality.

It is enough to show that $v_0$ is a super-solution, since then the claim follows by minimality of $u_*$,
and inequality $v_0 \leq u_*$ which is due to definition of $v_0$. 
Let $x\in V_*$  be a point of touch, i.e. $u_*(x) = v_0(x) $; we claim that at these points $v_0$ is stable.
Indeed, we have
\begin{align*}
\Delta^1 v_0(x) \leq \Delta^1 u_*(x) &\leq m \mathbb{I}_{ \left \{ 0<u_*\leq \frac 1m n^{2/d} \right\} } (x) - \mu_0(x) \\ &= 
m \mathbb{I}_{ \left\{ 0< v_0\leq \frac 1m n^{2/d} \right \} } (x) - \mu_0(x),
\end{align*}
where we have used that $u_*(x) = v_0(x)$ and hence the characteristic functions coincide at $x$.

Now assume that $x$ is not a point of touch, meaning that $ v_0(x) < u_*(x)$, and assume further that the sandpile configuration $\mu$ corresponding to $v_0$ is not stable at $x$.
This means that either $\mu (x)>m$, or we have both $\mu (x)> 0$ and $v_0(x)>\frac 1m n^{2/d} $.
Since $v_0(x)< u_*(x)$, we see that in both cases we could have toppled more mass from $x$, which means that in the toppling
process we had not moved the maximal allowed excess mass, contradicting the toppling rule.
We conclude that there can be no $x\in \Z^d$ which is not stable and at the same time 
$v_0(x)<u_*(x)$, hence the configuration corresponding to $v_0 $ is stable, and the proof of the lemma is complete.
\end{proof}


\section{Analysis of discrete shapes for single sources}\label{sec-single-source}

In this section we analyse the shape of the sandpile in the discrete space $\Z^d$ generated by initial distribution $n\delta_0$
and threshold parameter $m$. Namely, we establish a certain discrete monotonicity property of the odometer function,
estimate the size of regions in the set of visited sites with and without mass,
and conclude the section by proving discrete Lipschitz (uniform with respect to the threshold $m$),
and $C^{1,1}$ regularity estimates on the odometer function.

\subsection{Discrete monotonicity}
As in the case of \textbf{BS} \cite{AS}, here as well the odometer function of the singularly perturbed sandpile 
started with initial mass concentrated at a single point of $\Z^d$,
enjoys monotonicity in lattice directions. More precisely,
denote by $\mathcal{S}$ the set of mirror symmetry hyperplanes of the unit cube $[0,1]^d$;
clearly $\mathcal{S}$ consists of the  hyperplanes
$\{x_i = 1/2\}$, $\{x_i = x_j\}$ and $\{x_i=-x_j\}$ where $i,j=1,...,d$ and $i\neq j$.
In particular, there are $d^2$ elements in $\mathcal{S}$.

The next theorem is the analogue of what we had for \textbf{BS}
(the result also holds true for the classical Abelian sandpile as is proved in \cite{AS}).
The proof of the current version, is a simplification of the one we had in \cite{AS}.
We will follow the same line of arguments as we had in \cite[Theorem 4.5]{AS},
and we will only outline the differences for this case, which are minor.

\begin{theorem}\label{Thm-monotonicity}{\normalfont{(Directional monotonicity)}}
For initial mass distribution $\mu_0 = n \delta_0$
let $V\subset \Z^d$ be the set of visited sites, and let $u$ be the odometer function.
Fix any hyperplane $T\in \mathcal{S}$.
Then, for any $X_1, X_2 \in \Z^d$, such that $X_1 - X_2$ is a non-zero vector orthogonal to $T$,
we have
$$
u(X_1) \geq u(X_2) \text{ if and only if } |X_1| \leq |X_2|.
$$
\end{theorem}

\begin{proof}
Translate $T$ to a position where it has equal distance from 
$X_1$ and $X_2$. Let $T_0$ be this unique translated copy of $T$.
For a given $x\in \R^d$ denote by $x^*$ its mirror reflection with respect to the hyperplane $T_0$.
Due to the choice of $T$ and $X_1$, $X_2$ it follows that 
this reflection preserves the lattice, and we 
we have $(\Z^d)^* = \Z^d$, in particular  $X_1 = (X_2)^*$.

Set $V^* := \{ x^*: \ x\in V \}$ for the reflection of $V$, 
and similarly define the reflected odometer function by $u^*(x) : = u(x^*)$, where $x\in \Z^d$.
From the discussion above, we have that $u^*$ is defined on $\Z^d$ and $V^*\subset \Z^d$.
Let $\mathcal{H}_-$ be the closed halfspace of $\R^d$ determined by $T_0$ and containing the origin,
and let $\mathcal{H}_+ = \R^d \setminus \mathcal{H}_- $.
Consider the set $V_T = V_{-} \cup V_+$ where
$V_{-} :=  \mathcal{H}_- \cap V $ and $V_+ := \mathcal{H}_+ \cap V \cap V^*  $.
Our first goal is to show that $V_T = V$.
By definition $V_T \subset V$, so we need to establish the reverse inclusion.
Consider the function
$$
u_T (x) =  \begin{cases}  u(x) , &\text{if $x \in V_-$}, \\ 
\min\{ u(x), u^* (x) \}   ,&\text{if  $  x \in V_+   $}.  \end{cases}
$$
We claim that $u_T$ defines a super-solution in a sense of Definition \ref{def-super-sol}.
To see this, consider two cases.

\noindent \textbf{Case 1.} $x\in V_-$. 
As in \cite[Theorem 4.5]{AS}, here as well we get $ \Delta^1 u_T (x) \leq \Delta^1 u(x) $,
which, together with the equality $u_T(x) = u(x)$ gives
$$
\Delta^1 u_T(x) \leq m \mathbb{I}_{ \left\{ 0<u_T\leq \frac 1m n^{2/d} \right\} } (x) - n \delta_0.
$$

\vspace{0.1cm}

\noindent \textbf{Case 2.} $x\in V_+$. 
If $u_T(x) = u(x)$ then we proceed as in Case 1, and get the stability at $x$.
We will thus assume that $u_T(x) = u(x^*)$.
In this case, following again \cite[Theorem 4.5]{AS}, we obtain $\Delta^1 u_T (x) \leq \Delta^1 u(x^*)$.
Observe that if $u_T(x) \leq \frac 1m n^{2/d}$ then the last inequality gives the stability of the sandpile at $x$.
Otherwise, if $u_T(x) > \frac 1m n^{2/d}$ we get $u(x_*)>\frac 1m n^{2/d}$ and hence
$$
\Delta^1 u_T (x) \leq \Delta^1 u(x^*) \leq -n \delta_0(x_*) \leq 0,
$$
completing the proof that $u_T$ satisfies condition (ii) of being a super-solution.

\smallskip

Part (i) follows from the fact that the support of $u_T$ is  included in the support of $u$.
We conclude that $u_T$ is a super-solution, and hence due to the minimality of Lemma \ref{Lem-equality-u-star},
we get $u_T=  u$
and in particular,
$$
u(X_2) \leq  u_T(X_2) = \min\{ u(X_1), u(X_2) \} \leq u(X_1),
$$
completing the proof of the theorem.
\end{proof}

\subsection{The size of the sandpile shapes}
There are two distinct
regions in the visited set of the sandpile process. 
Namely, the one which is free of mass, and the one where each vertex
carries mass equal to the threshold $m$.
The aim of this subsection, is to give estimates on the size of these
two regions. 

\smallskip

Recall, that we are dealing with sandpiles with initial distribution $n\delta_0$
and mass threshold equal to $m$. We will be using the following notation (cf. Figure \ref{Fig-many}):
\begin{itemize}
\vspace{0.1cm}
 \item $V_{n,m} \subset \Z^d \ - $ the set of visited sites of the sandpile,
 \vspace{0.1cm}
 \item $V_{n,0}\subset V_{n,m} \ - $ the set where the odometer $u> \frac 1m n^{2/d}$; in particular $u$ is harmonic on $V_{n,0}\setminus \{0\}$ and (thanks to the discrete monotonicity of $u$)
 the origin lies in $V_{n,0}$,
 \vspace{0.1cm}
 \item $V_{n,1} : = V_{n,m}\setminus V_{n,0} \ - $ the set where the entire mass of the sandpile is concentrated.
\end{itemize}

\smallskip

\begin{lem}\label{Lem-non-degeneracy}{\normalfont{(Upper bound on the size of mass region)}}
Let $x_0 \in V_{n,1}$ and $r>0$ be such that $B(x_0, r) \cap V_{n,0} = \emptyset$.
Then, there exists a dimension dependent constant $C_d>0$ such that
for any $n>0$ large and any $1<m<n$ one has
$$
r\leq C_d \frac 1m n^{1/d}.
$$
\end{lem}

\begin{proof}
Consider the function $w(x) = u(x) - u (x_0) -  m |x-x_0|^2$,
for which we have $\Delta^1 w(x) =0 $ for all $x\in B(x_0,r) \cap V_{n,1}=:D $, and hence by DMP
we get 
\begin{equation}\label{a1}
0 = w (x_0)  \leq   \max\limits_{ \partial D } w .
\end{equation}
Observe that
$$
\partial D   = (\partial B(x_0,r)  \cap V_{n,1} ) \cup ( B(x_0, r) \cap \partial V_{n,m} ) := \Gamma_1 \cup \Gamma_2,
$$
where we have used the fact that the ball $B$ has no intersection with the region $V_{n,0}$.
Since $u=0$ on $\partial V_{n,m}$ we have $w<0$ on $\Gamma_2$ which, together with \eqref{a1} 
implies
$$
0\leq \max_{\Gamma_1} w(x) \leq  \max\limits_{\Gamma_1} u(x) - u (x_0)    - c_d  m r^2 \leq \frac 1m n^{2/d}  - 
c_d  m R^2,
$$
where $c_d>0$ is some small constant depending on dimension $d$ only.
Rearranging the last inequality completes the proof of the lemma.
\end{proof}

\begin{lem}\label{Lem-Lip-1}
Fix $r>0$ large, and a function $f:B \to \R_+$  bounded above by $m$,
where $B = B(0,r) \cap \Z^d $.
Let $u$ be the (unique) solution to
$$
\Delta^1 u = f \text{ in } B \text{ and } u=0 \text{ on } \partial B.
$$
Then, we have
\vspace{0.1cm}
\begin{itemize}
 \item[\normalfont (i)] $|u(x) - u(y)| \leq C r m $, for any $x\in B_{r/2} = B(0,r/2) \cap \Z^d$ and $y\sim x$,
 \vspace{0.1cm}
 \item[\normalfont (ii)] $ - r^2 m \leq u(0) \leq 0 $,
\end{itemize}
with dimension dependent constant $C$.
\end{lem}

\begin{proof}
We start with (i). For $x\in B$ we have the representation
\begin{equation}\label{g1}
u(x) = \sum\limits_{w \in B} G(x,w) f(w), 
\end{equation}
where $G$ is the Green's kernel for $B$, i.e. 
the expected number of passages through $w$ of a random walk started at $x$
before escaping from $B$. More precisely, we have
\begin{equation}\label{g2}
G(x,w) = g(x-w)  - \mathbb{E}^x g(X_T - w),
\end{equation}
where $T$ is the first exit time from $B$ of the walk, $g$ is the fundamental solution
of $\Delta^1$ defined in subsection \ref{sub-sec-discrete-Lap},
and $\mathbb{E}^x$ stands for the expectation conditioned that the walk has started from $x$.

Now fix any $x,y \in B_{r/2}$ such that $x\sim y$.
Then, from \eqref{g1} we have
\begin{multline}\label{u1}
|u(x) - u(y)| \leq  | G(x,x) - G(y,x)  | f(x) + | G(x,y) - G(y,y)  | f(y) + \\
\sum\limits_{w\in B\setminus \{x,y\}} | G(x,w) - G(y,w) | f(w). 
\end{multline}
From \eqref{g2} we get
\begin{multline*}
|G(x,w) - G(y,w)| \leq |g(x-w) - g(y-w)| + \sum\limits_{z\in \partial B} \mathbb{P}^w [ X_T = z ] \left|  g(z- y) - g(z-x) \right| \leq \\
\frac{C}{|x-w|^{d-1}} + C \sum\limits_{z\in \partial B} \frac{ \mathbb{P}^w [ X_T = z ] }{|z-x|^{d-1}}  \leq \frac{C}{|x-w|^{d-1}} +  \frac{C}{r^{d-1}},
\end{multline*}
where we have used the asymptotics \eqref{g-x-y} and that $x\sim y $ to estimate the difference of $g$-s, and the fact that $x\in B_{r/2}$ 
in addition to those to bound the sum. With this estimate getting back to sum in \eqref{u1}
we obtain
$$
\sum\limits_{w\in B\setminus \{x,y\}} | G(x,w) - G(y,w) | f(w) \leq m \sum\limits_{w\in B\setminus \{x\} } \left( \frac{1}{|x-w|^{d-1}} + \frac{1}{r^{d-1}} \right) \leq C_d m  r,
$$
where the sum of the first term is estimated by a simple counting argument relying on the structure of $B$ (see, e.g. \cite[Lemma 5.2]{Lev-Per}), while bound on the sum involving the second summand
follows from a trivial estimate $|B| \leq C_d r^d$.
Returning to \eqref{u1} we are left to estimate only the first two sums on the \emph{r.h.s.},
 but their contribution is bounded
above by $C_d m$ in view of the representation \eqref{g2}, the fact that $x\sim y$, and then using that $g$ is symmetric with respect to coordinate axes and has
bounded Laplacian.

\smallskip

We next proceed to the claim of (ii). The upper bound is a consequence of DMP; to establish the lower bound,
consider the function 
$$
v(x) = r^2 m - m|x|^2, \qquad x\in \Z^d.
$$
Clearly, $\Delta^1 v = -m$ everywhere and $v\leq 0$ on $\partial B$. From here,
we get
$$
\Delta^1 (u+v) = f- m \leq 0 \text{ in } B \text{ and } u+v = v \leq 0 \text{ on } \partial B,
$$
and hence
$$
\min_{B } (u+v) \geq \min_{\partial B} (u+v) = \min_{\partial B} v = r^2 m - m \max_{\partial B} |x|^2.
$$
We thus have
$$
u(0) \geq -v(0) + r^2 m - m \max_{\partial B} |x|^2 \geq -m r^2.
$$
The proof of the lemma is now complete.
\end{proof}

A trivial corollary of the 1-step Lipschitz estimate of Lemma \ref{Lem-Lip-1}
is the following.

\begin{cor}\label{cor-to-Lip-1}
Retain all notation and conditions of Lemma \ref{Lem-Lip-1}. Then, for any $x,y\in B_{r/2}$ one has
\begin{itemize}
 \item[\normalfont (i)] $|u(x) - u(y)| \leq C r  m |x-y|$ for any $x,y\in B_{r/2}$,
 \vspace{0.1cm}
 \item[\normalfont (ii)] $- C r^2 m \leq u(x) \leq 0$ for any $x\in B_{r/2}$,
\end{itemize}
where $C>0$ is a constant depending on dimension only.
\end{cor}

\begin{proof}
To see (i), take any path through $B_{r/2}$ connecting $x$ and $y$, namely
$$
x := x_0 \sim x_1\sim...\sim x_k : = y,
$$
where $x_i\in B_{r/2}$ for $0\leq i \leq k$. Clearly, we can assume $k \asymp |x-y|$, by considering a path of the shortest length.
Now, (i) follows from
$$
|u(x) - u(y)| \leq \sum\limits_{i=0}^{k-1} | u(x_{i+1}) - u(x_i) | \leq C r  m k \leq C r m |x-y|,
$$
where we  have applied Lemma \ref{Lem-Lip-1} to each summand.

For (ii) observe that the upper bound is again due to DMP, and the lower bound
is in view of 
$$
|u(x) | \leq |u(0) - u(x)| + |u(0)| \leq C m r |x |  + C m r^2 \leq C m r^2,
$$
where we have used Lemma \ref{Lem-Lip-1} (ii).
The proof is now complete.
\end{proof}

The next result will be used in estimating the size of the mass region $V_{n,1}$.

\begin{lem}\label{Lem-quad-bound}{\normalfont{(Quadratic upper bound)}}
Let $x_0 \in \partial V_{n,m}$ be any. 
Take $r>0$ large so that the ball $B_r = B(x_0, r) \cap \Z^d$ does not contain the origin
of $\Z^d$. Then 
$$
u(x) \leq C m r^2, \text{ for all } x\in B_{r/2} .
$$
\end{lem}

\begin{proof}
We start with the splitting $u=u_1 + u_2$ where
\begin{equation}
 \Delta^1 u_1 = \Delta^1 u  \text{ in } B_r \text{ and } u_1=0 \text{ on } \partial B_r,
\end{equation}
and 
\begin{equation}
 \Delta^1 u_2= 0 \text{ in } B_r \text{ and } u_2 = u \text{ on } \partial B_r.
\end{equation}

Since the origin does not lie in $B_r$, we have $0\leq \Delta^1 u \leq m$ in $B_r$,
and hence applying Lemma \ref{Lem-Lip-1} (i) (to the function $u_1(x + x_0)$) we get $|u_1(x_0)| \leq C m r^2 $,
consequently
$$
u_2 (x_0) = u(x_0) - u_1(x_0) = -u_1(x_0) \leq C m r^2.
$$
From here and using that fact that $u_2$ is harmonic in $B_r$ and non-negative on the boundary of $B_r$,
we get, by discrete Harnack (see \cite[Theorem 1.7.2]{Lawler-book-walks}), that
$$
u_2(x) \leq C m r^2 \text{ for all } x\in B_{r/2}.
$$
Since $u_1\leq 0$ in $B_r$ by DMP, the last inequality implies the claim of the current lemma.
\end{proof}

Recall that in Lemma \ref{Lem-non-degeneracy} we established an upper bound
on the thickness of the mass-region. We are now in a position to also  give  a bound from below.

\begin{lem}\label{Lem-size-of-mass-region}{\normalfont{(The size of mass-region)}}
Let $x_0\in \partial V_{n,m}$ and let $r>0$ be such that $B(x_0,r) \subset V_{n,1}$,
and $B(x_0, r+1) \cap V_{n,0} \neq \emptyset$.
Then
$$
r  \asymp  \frac{n^{1/d}}{m},
$$ 
where equivalence holds with dimension dependent constants.
\end{lem}
\begin{proof}
The upper bound on $r$ is due to Lemma \ref{Lem-non-degeneracy}, and we only need to prove the lower bound here.
By assumption there exists $z \in B_{r+1}  \cap \Z^d $ such that $u(z)\geq  \frac{n^{2/d}}{m}$.
Now, applying Lemma \ref{Lem-quad-bound} we obtain
$$
\frac{n^{2/d}}{m} \leq u(z) \leq C m (r+1)^2,
$$
which completes the proof of the lemma.
\end{proof}

\begin{prop}\label{prop-size-of-mass-FREE-region}{\normalfont{(The size of mass-free region)}}
There exist dimension dependent positive constants $c_1<c_2$ such that
for any fixed threshold $m>1$ and $n$ large enough one has
$$
B(0, c_1 n^{1/d} ) \cap \Z^d  \subset V_{n,0} \subset B(0, c_2 n^{1/d} ) \cap \Z^d .
$$
\end{prop}
\begin{proof}
We first compare $V_{n,0}$ with sets having a very simple structure.
Let $R>0$ be the largest integer such that the point $X_R=(0,...,0,R) \in \Z^d$ is inside $V_{n,0}$
but is not an interior point of the set $V_{n,0}$.
In particular, $u(X_R)> \frac1m n^{2/d}$ but for some $y\sim X_R$ we have 
$u(y) \leq \frac 1m n^{2/d}$.
Let $\mathcal{S}_R$ be the simplex with vertices at $\pm R e_i$, where $i=1,2,...,d$,
in other words $\mathcal{S}_R$ is the ball of radius $R$ in $l_1$ metric.
Let us prove the following inclusions:
\begin{equation}\label{incl2}
 \mathcal{S}_R \subset V_{n,0} \subset [-R, R]^d.
\end{equation}
Notice, that both the simplex and the cube in \eqref{incl2} are restricted to $\Z^d$.

We take any $x\in \mathcal{S}_R$ and show that it is in $V_{n,0}$.
Following \eqref{od-symm}, we know that the odometer function is symmetric with respect
to coordinate axes, and hyperplanes through the origin with normals
in the directions $e_i\pm e_j$, and obviously so is the simplex $\mathcal{S}_R$,
and hence, it will be enough to prove \eqref{incl2}
for $X_0=(x_1,...,x_d)\in \Z^d$ satisfying $x_d \geq |x_i|$, $i=1,...,d-1$.
To accomplish that, we will use the directional monotonicity of the odometer.

For $1\leq i \leq d-1$, set $\nu_i = e_d - e_i$,  $\nu_{i+d-1}=e_d + e_i$, and denote $\nu_{2d-1}= e_d$.
For each $1\leq i \leq 2d-1$ consider the discrete halfspace
$$
H_i = \{ X\in \Z^d:  \ X \cdot \nu_i \geq 0 \},  
$$
and define $H_* : = H_1\cap...\cap H_{2d-1}$.
In view of the choice of the vectors $\nu_i$ we have 
$$
H_* =\{X=(x_1,...,x_d)\in \Z^d: x_d \geq |x_i|, \ i=1,2,...,d-1 \},
$$
in particular $X_0, X_R \in H_*$. Next, consider the cone 
$$
\mathcal{C}_0 = \{ t_1 \nu_1 +... + t_{2d-1} \nu_{2d-1}: \ t_i \in \Z_+, \ 1\leq i \leq 2d-1 \}.
$$
Since $\nu_i \cdot \nu_j \geq 0$ for all $1\leq i,j\leq 2d-1$, we get $\mathcal{C}_0 \subset H_*$.
We now translate $\mathcal{C}_0$ to $X_0$ by setting $ \mathcal{C} = X_0 +\mathcal{C}_0  $,
and as $X_0 \in H_*$ we get $\mathcal{C} \subset H_*$.
In view of the choice of the collection $\{\nu_i\}_{i=1}^{2d-1}$
and the points $X_0, X_R$, it is easy to see $X_R\in \mathcal{C}$.
Finally, in the cone $\mathcal{C}$ we use the directional monotonicity given by Theorem \ref{Thm-monotonicity} 
which implies that in $\mathcal{C}$ the odometer $u$ attains its maximum at the vertex of the cone, i.e. at $X_0$.
Since $X_R\in \mathcal{C}$ we obtain $u(X_0) \geq u(X_R)$, and hence $X_0\in V_{n,0}$ which completes the proof of the first inclusion of \eqref{incl2}.
The second inclusion, is much simpler, and follows easily by using monotonicity in the coordinate directions only.

\smallskip

With \eqref{incl2} at hand, the proof of the proposition will be complete, once we show that
$R \asymp n^{1/d}$ with constants depending only on $d$. In what follows we prove this, and hence the proposition.
We will use the fact that $\partial V_{n,0}$ is locally a graph, along with bounds of the size of the mass-region.
Here again, due to the symmetry, it will be enough to consider the region where $x_d \geq |x_i|$ for all $1\leq i \leq d-1$.

Let $\Pi \subset \Z^{d-1} \times \{0\}$ be the projection of the set $H_*\cap \mathcal{S}_R$
onto $\Z^{d-1}\times \{0\}$. In view of the monotonicity of $u$ in the direction of $e_d$,
we have $\Pi \subset V_{n,0}$. Moreover, definitions of $H_*$ and $\mathcal{S}_R$
imply that $|\Pi| \asymp R^{d-1}$ with constants in the equivalence depending on dimension $d$ only.
Now, for a given $X= (\overline{x}, 0)\in \Pi$ where $\overline{x}\in \Z^{d-1}$,
let $t_1\in \Z_+$ be the smallest integer such that $X_1= (\overline{x}, t_1) \notin V_{n,0}$
and let $t_2 \geq t_1$ be the smallest integer such that $u$ vanishes at $X_2 =(\overline{x}, t_2) $.
In view of the choice of $X_1$ and the discussion above, the cone $X_1 + \mathcal{C}_0$
has no points of $V_{n,0}$. Moreover, $X_2 \in \mathcal{C}_0$ and according to Lemma \ref{Lem-size-of-mass-region}
we obtain that $t_2 - t_1 \asymp r$ where $r=\frac 1m n^{1/d}$ and constants in the equivalence depend on dimension only. We thus see that each line in the direction of $e_d$ through points of $\Pi$ intersects the
mass-region by $\asymp r$ points, and since each point in $V_{n,1}$ carries mass $m$
and the total mass of the system is preserved, we get
$$
R^{d-1} r m \asymp n,
$$
from this, and the estimate of Lemma \ref{Lem-size-of-mass-region}, we get
$R \asymp n^{1/d}$ completing the proof of this proposition.
\end{proof}

\begin{remark}
Observe, that putting together Lemma \ref{Lem-size-of-mass-region} and Proposition \ref{prop-size-of-mass-FREE-region}
we see that for large $n>1$ and $1<m<n^{1/d}$, the set of visited sites of the sandpile grows proportionally to $n^{1/d}$
uniformly in $m$.
\end{remark}

\begin{remark}\label{rem-graph}
The proof of Proposition \ref{prop-size-of-mass-FREE-region} shows, that in the cone 
$$
\{ x=(x_1,...,x_d) \in \Z^d: \ x_d \geq |x_i|, \ i=1,..,d-1 \} , 
$$
the set $\partial V_{n,m}$ is a graph over $\Z^{d-1}\times \{0\}$. Namely, any line $(\overline{x},t)\in \Z^{d-1} \times \Z_+$ in
this region intersects $\partial V_{n,m}$ by a single point.
Thanks to the symmetry of the sandpile, this property also holds in the regions
where the $i$-th coordinate of a point is the largest, for any $1\leq i \leq d$ .
\end{remark}

\subsection{Uniform Lipschitz estimate}
In this section we prove that away from the origin the odometer function is Lipschitz uniformly in $m$,
with the Lipschitz constant bounded above by the size of the model.

Assume $B_r$ is a ball of radius $r$ centred at the origin, and let $f:\overline{B_r}\cap \Z^d \to \R$
be harmonic in $B_r$. The following estimate for the discrete derivative of $f$ (``difference estimate``)
is proved in \cite[Theorem 1.7.1]{Lawler-book-walks}; there exists a constant $C>0$ independent of $f$ and $r$
such that $|f(x) - f(0)| \leq C r^{-1} ||f||_{l^\infty}$ for any $x\sim 0$.
Iterating this bound as in the proof of Corollary \ref{cor-to-Lip-1} we get
\begin{equation}\label{Lip-est-Poisson-part}
 |f(x) - f(y)| \leq C  |x-y| r^{-1} ||f||_{l^\infty},  \ \ \forall x,y\in B_{r/2}.
\end{equation}

\bigskip

\begin{prop}\label{prop-Lip}
Let $u$  be the odometer for the sandpile with initial distribution $n\delta_0$ and threshold $m$.
Then, for any $r_0>0$ small there exists a constant $C=C(r_0,d)$ such that
$$
|u(x) - u(y) | \leq C n^{1/d} |x-y|,
$$
for any $x,y \in \Z^d\setminus B(0, r_0 n^{1/d} )  $.
\end{prop}

\begin{proof} The idea is to show that $u$ is Lipschitz in a neighbourhood of the mass-region $V_{n,1}$,
and also in the neighbourhood of the origin. Then, we can conclude the Lipschitz estimate in between
these two regions (where $u$ is harmonic) by DMP. For the clarity, we will split the proof into a few steps.

\smallskip

\textbf{Step 1.} \emph{Lipschitz bound near the mass-region}

Fix any $x_0\in \partial V_{n,m}$ and take $r>0$ such that the ball $B = B(x_0,r)\cap \Z^d$ does not contain the origin.
Next, we write $u=u_1 + u_2$ in $B$ where 
$$
\Delta^1 u_1 = \Delta^1 u \text{ in } B \text{ and }  u_1 = 0 \text{ on } \partial B_r,
$$
and
$$
\Delta^1 u_2 = 0 \text{ in } B \text{ and } u_2 = u  \text{ on } \partial B_r.
$$
Since $ 0=u(x_0) = u_1(x_0) + u_2(x_0)$, by Lemma \ref{Lem-Lip-1} we get
$$
0\leq u_2(x_0) \leq C r^2 m.
$$
The latter, in view of discrete Harnack, implies that $u_2(x) \leq C r^2 m$ for all $x\in B_{r/2}$,
with $C>0$ depending on dimension only. This bound on $u_2$ and
estimate \eqref{Lip-est-Poisson-part} imply
$$
|u_2(x) - u_2(y)| \leq C  m r |x-y|   \text{ for any } x,y\in B_{r/4}.
$$
The last estimate coupled with Corollary \ref{cor-to-Lip-1} applied to $u_1$, gives
\begin{equation}\label{a--1}
 |u(x) - u(y)| \leq C m r |x-y|, \text{ for any } x,y\in B_{r/4}.
\end{equation}

Recall, that the thickness of the mass-region is bounded above by $C_0 n^{1/d} m^{-1}$ according to Lemma \ref{Lem-size-of-mass-region}.
In view of Proposition \ref{prop-size-of-mass-FREE-region} we see that the ball $B=B(x_0,r)$ with $r=16C_0 n^{1/d} m$ and $x_0\in \partial V_{n,m}$
does not contain the origin, in particular we have $0\leq \Delta^1 u\leq m$ in $B \cap \Z^d$.
With this in mind, we take $r=16 C_0 n^{1/d} m$ and varying $x_0$ on the boundary of $\partial V_{n,m}$,
from we \eqref{a--1} we get
\begin{equation}\label{a0}
 |u(x) - u(y)| \leq C n^{1/d} |x-y|, \ \  \forall x,y\in \{ z\in \Z^d:  \ \textrm{dist}(z,V_{n,1} ) \leq r/4 \} : = V_*.
\end{equation}

\smallskip

\textbf{Step 2.} \emph{Lipschitz bound near the origin}

Here again we will partition the solution into two parts, namely one with bounded Laplacian, and another one as the Green's kernel.
We write $u=u_1 + u_2$ where 
$$
\Delta^1 u_1 = \Delta^1 u - n\delta_0 \text{ in } V_{n,m} \text{ and } u_1= 0 \text{ on } \partial V_{n,m},
$$
and hence
$$
\Delta^1 u_2 =  - n\delta_0 \text{ in } V_{n,m} \text{ and } u_2= 0 \text{ on } \partial V_{n,m}.
$$
Notice that $u_2$ is the Green's function of $V_{n,m}$ with the pole at the origin, multiplied by $n$.
From Proposition \ref{prop-size-of-mass-FREE-region} and \cite[Lemma 5.1]{AS} we have
\begin{equation}\label{u2-0}
 |u_2(x) - u_2(y)| \leq C_{r_0} n^{1/d}, 
\end{equation}
for any  $x\in V_{n,0}$ with $r_0 n^{1/d} \leq |x| \leq 2 r_0 n^{1/d}$  and $y\sim x$.
For $u_1$, we take any $x\in V_{n,0}$ satisfying $r_0 n^{1/d} \leq |x| \leq 2 r_0 n^{1/d}$ and using the fact that $u_1$
has bounded Laplacian everywhere, from Lemma \ref{Lem-Lip-1}, choosing the radius of the ball to be $r_0 n^{1/d} m^{-1}$, we obtain
\begin{equation}\label{u1-0}
 |u_1(x) - u_1(y)| \leq C n^{1/d}, 
\end{equation}
where $y\sim x$. Combining \eqref{u2-0} and \eqref{u1-0} leads to Lipschitz estimate 
\begin{equation}\label{a2}
 |u(x) - u(y)| \leq C n^{1/d},
\end{equation}
where $x\sim y$, $x\in V_{n,0}$ with $r_0 n^{1/d} \leq |x| \leq 2 r_0 n^{1/d}$.

\smallskip

\textbf{Step 3.} \emph{Interpolating between two regions}

Consider the set 
$$
E: = \{x\in V_{n,0}: \ |x|\geq r_0 n^{1/d} \text{ and } \mathrm{dist}(x, \partial V_{n,1})>1 \} .
$$ 
It is left to establish a Lipschitz estimate for $u$ in $E$. Take any direction vector  $\pm e_i$
and consider the function $w(x) = u(x+e_i ) - u(x)$ with $x\in E$.
Clearly $\Delta^1 w =0$ in $E$, and in view of estimates \eqref{a0} and \eqref{a2} we have $|w|\leq C_{r_0} n^{1/d}$
on $\partial E$. Invoking DMP we obtain $|w| \leq C_{r_0} n^{1/d}$ on $E$.
This shows 1-step Lipschitz bound for $u$. Iterating it as in Corollary \ref{cor-to-Lip-1}
leads to
$$
 |u(x) - u(y)| \leq C_{r_0} n^{1/d} |x-y|, \ \ \forall x,y\in E.
$$
This bound, combined with \eqref{a0} completes the proof of the proposition.
\end{proof}

\subsection{$C^{1,1}$ estimates}
Here we prove further regularity estimates for the odometer function.
The goal is to show that discrete derivatives, as defined in subsection \ref{sub-sec-discrete-Lap},
are Lipschitz away from the origin, however here the Lipschitz constant would depend on $m$.
We will see later, in Section \ref{sec-scaling-lim}, that these estimates, when properly scaled,
force the gradient of odometer function to vanish 
on the boundary of (any) scaling limit of the set of visited sites.

\begin{prop}\label{prop-1-step-C11-log}
Let $u$  be the odometer for the sandpile with initial distribution $n\delta_0$ and threshold $m$,
where $m\geq 1$ is fixed and $n>1$ is large.
Then, for any $r_0>0$ small there exists a constant $C=C(r_0,d)$ such that
for any unit vector $e \in \Z^d$, and any $x,y\in \Z^d$ satisfying $|x|, |y|\geq r_0 n^{1/d}$,
we have
$$
|\nabla_e^1    u(x)  - \nabla_e^1   u(y) | \leq C m   |x-y|.
$$
\end{prop}

\begin{proof}
For $x\in \Z^d$ denote $w(x) = \nabla^1_{e}  u(x)$.
We fix a radius $r>0$, which is a large constant independent of $n$ and $m$,
a point $x_0\in \Z^d$ with $|x_0|\geq r_0 n^{1/d}$,
and consider the problem in the ball $B= B(x_0, r) \cap \Z^d$.
With this notation, we need to show that
$$
 | \nabla^1_{\widetilde{e}} w(x) | \leq C m, \ \ \forall x\in B_{r/2},
$$
for any unit vector $\widetilde{e}\in \Z^d$,
since then the estimate of the proposition will follow by iteration, as in 
Corollary \ref{cor-to-Lip-1} for instance.
We now fix a unit vector $\widetilde{e}\in \Z^d$ and suppress it from the subscript of $\nabla^1_{\widetilde{e}}$.

In view of the choice of $r$ and Proposition \ref{prop-Lip}, we have
\begin{equation}\label{b1}
| w(x) | \leq C n^{1/d}, \qquad \forall x\in B_{r/2},
\end{equation}
with a constant $C=C(r_0, d)$. As in the proof of Proposition \ref{prop-Lip}, 
we do the splitting $w= u_1 + u_2$ in $B$, where $u_1$ is the potential
part of $w$ and $u_2 $ is harmonic.
By DMP, estimates \eqref{b1}, and \eqref{Lip-est-Poisson-part} we have 
\begin{equation}\label{b2}
|\nabla^1 u_2(x)| \leq C  r^{-1} \max_{\partial B} |w|  \leq C  \ \ \text{ in } B_{r/4}.
\end{equation}
Thus, it is left to handle the part with $u_1$. To this end consider the set
$$
E:= \{ x\in B: \ \mathrm{dist}(x, \partial V_{n,m})\leq 1  \} \cup 
\{ x\in B:  \ \mathrm{dist}(x, \partial V_{n,0}) \leq 1 \},
$$
which is the $1$-discrete neighbourhood of the boundaries of $V_{n,m}$ and $V_{n,0}$
contained in the ball $B(x_0,r)$. Observe, that $\Delta^1 w(x) =0$ on $B\setminus E$.
Now, using the Green's representation of $u_1$ as in the proof of Proposition \ref{prop-Lip},
for all $x\in B_{r/2}$ we get
\begin{multline}\label{b3}
| \nabla^1 u_1 (x) | \leq || \Delta^1 w(x)  ||_{L^\infty(B)} \sum_{z\in E} | \nabla_x^1 G(x,z) | \lesssim \\
m \sum_{z\in E ,  |x-z|\geq 1} \left( \frac{1}{|x-z|^{d-1}} + \frac{1}{r^{d-1}}  \right) \lesssim m,
\end{multline}
where the penultimate inequality is proved in the proof of Proposition \ref{prop-Lip},
while the last one, with a constant depending on $r$ comes from the
estimate on the number of points of $E$.

Putting together \eqref{b2} and \eqref{b3} we complete the proof of this proposition.
\end{proof}

\begin{remark}
It should be remarked that some of the results and approaches used in this section,
such as Lemmas \ref{Lem-non-degeneracy}, and \ref{Lem-quad-bound},
as well as Propositions \ref{prop-Lip}, and \ref{prop-1-step-C11-log}
go in parallel with the theory of free boundary problems of the form \eqref{eq-odometer2}
(with a more general \emph{r.h.s.}) in continuous space (see the first chapter of \cite{PSU} for instance).
Although our approach in this section shares some similarities with the continuous analogues,
the proofs however are much different due to the fact that we are working in a discrete space here.
\end{remark}


\section{Scaling limits}\label{sec-scaling-lim}

Here we prove that the sandpile shapes, generated from initial distribution concentrated
at a single vertex of $\Z^d$, have a scaling limit (which is a ball), when the threshold $m$ is fixed,
and the mass $n$ tends to infinity.
Then, we show that there is also a scaling limit as $m$ tends to infinity along with $n$
but very slowly with respect to $n$. That shape is still a ball, but the entire mass is concentrated on the boundary.

\bigskip

We will need a few notation.
For $n\geq 1$ set $h= n^{-1/d} $, and define the scaled odometer by
$u_h (x) =h^2 u_n(h^{-1} x) $ where $x\in h \Z^d$.
Next, for $0<h \leq 1$ and $\xi =(\xi_1,...,\xi_d)\in h\Z^d$ define the half-open cube
\begin{equation}\label{cube-def}
\mathrm{C}_h(\xi) = \left[ \xi_1 - \frac{h}{2}, \xi_1 +\frac{h}{2} \right)\times ... \times \left[ \xi_d - \frac{h}{2}, \xi_d +\frac{h}{2} \right).
\end{equation}
In order to study the scaling limit of the model,
we need to extend each $u_h$ to a function defined on $\R^d$.
We will use a standard extension of $u_h$ which preserves its discrete derivatives and hence $\Delta^h$-Laplacian.
Namely, for fixed $0<h\leq 1$ define a function $U_h :\R^d \to \R_+$, where for each $\xi \in h \Z^d$
and any $x\in \mathrm{C}_h(\xi)$ we have set $U_h(x) = u_h(\xi)$. 
Clearly for any $\xi \in h\Z^d$ and any $x \in \mathrm{C}_h(\xi)$ we get
\begin{equation}
 \nabla_e^h U_h (x) = \nabla_e^h u_h(\xi) \text{ for all unit } e\in \Z^d   \text{ and } \Delta^h U_h(x) = \Delta^h u_h (\xi).
\end{equation}

Applying the scaling to the estimates of Proposition \ref{prop-Lip} and Proposition \ref{prop-1-step-C11-log},
we obtain that for any $\rho>0$ there exists a constant $C_{\rho}$ independent of $m\geq 1$, such that
the following hold true:

\begin{itemize}
\item \textbf{Uniform Lipschitz estimate} 
\indentdisplays{1pt}
\begin{equation}\label{scaled-lip}
| u_h (x) - u_h(y) | \leq C_{\rho} |x-y|, \ \ \forall x,y\in h \Z^d\setminus B(0,{\rho}),
\end{equation}

\item \textbf{$C^{1,1}$-estimates} 
\indentdisplays{1pt}
\begin{equation}\label{scaled-grad-lip}
| \nabla_e^1 u_h (x) - \nabla^1_e u_h(y) | \leq C_\rho m |x-y|, \ \ \forall x,y\in h \Z^d\setminus B(0,\rho),
\end{equation}
where $e\in \Z^d$ is any unit vector.
\end{itemize}

\smallskip

\begin{theorem}\label{Thm-scaling-limit-fixed-M}{\normalfont(Scaling limit for fixed $m$)}
There exists a compactly supported non-negative function $u_0\in C(\R^d\setminus \{0\})$
which is spherically symmetric, and is $C^{1,1}$ outside any open neighbourhood of the origin, such that
\begin{itemize}
 \item[{\normalfont{(i)}}] $ u_{h} \to u_0 $ uniformly in $\R^d$ outside any open neighbourhood of the origin,
 \vspace{0.1cm}
 \item[{\normalfont{(ii)}}] $\Delta u_0 = 2d m \mathbb{I}_{ \{0< u_0 < \frac 1m \} }  - \delta_0$ in $\{u_0>0\}$
 in the sense of distributions, where $\delta_0$ is the Dirac delta at the origin
 and $\Delta$ denotes the continuous Laplace operator,
 \vspace{0.1cm}
\item[{\normalfont{(iii)}}] 
 $|\nabla u_0 | =0 $ everywhere on $\partial \{u_0>0\}$.
\end{itemize} 
\end{theorem}

\begin{proof}
We first show that there is a convergent subsequence of $\{U_h\}$ whose limit satisfies the requirements of the theorem.
Then, we conclude the proof by showing that any two convergent subsequences have the same limit.
For the proof of the convergent subsequence we will follow our approach from \cite[Theorem 5.3]{AS}, with
the only difference that here we also need to take care of the convergence of the gradient too. 

\smallskip

Fix $\rho>0$ small, and for $0<h \leq 1$ set 
$$
E_{\rho}(h)  = \{ x \in h\Z^d: \ |x| > \rho \}.
$$
Since the support of $U_h$ is uniformly bounded in $0<h\leq 1$
thanks to Lemma \ref{Lem-size-of-mass-region} and Proposition \ref{prop-size-of-mass-FREE-region},
from \eqref{scaled-lip} we get that $U_h$ is a bounded and Lipschitz function on $E_{\rho}(h)$, both uniformly in $0<h\leq 1$. We let $C_\rho>0$ be the Lipschitz constant of $U_h$ on $E_{\rho}(h)$, and define
\begin{equation}\label{ext-1}
 U_h^\rho(x) : = \inf\limits_{\xi \in E_{\rho}(h)} \big( u_h(\xi) + C_\rho |x-\xi| \big) , \qquad x\in \R^d,
\end{equation}
which, due to the mentioned properties of $U_h$, becomes a bounded and $C_\rho$-Lipschitz function on $\R^d$ which coincides with $U_h$ on the set $E_\rho(h)$.
The extension defined by \eqref{ext-1} is standard and well-known (see \cite[Theorem 2.3]{Heinonen} for instance).
Moreover, by construction we have that the family $\{U_h^\rho\}_{0<h\leq 1}$ is uniformly bounded and
is non-negative everywhere. Observe also that for $0<h\leq 1$, $\xi \in E_{2 \rho }(h)$
and $x \in \mathrm{C}_h(\xi)$ (see \eqref{cube-def}) by construction we have
\begin{equation}\label{u-h-U-h-est}
|U_h(x) - U_h^\rho(x)| = |U_h^\rho(\xi) - U_h^\rho(x)| \leq C_\rho h.
\end{equation}

\smallskip

We also do a similar extension for the discrete derivatives. 
Namely, for a unit vector $e\in \Z^d$ set $w_{h,e}(x) = \nabla_e^h u_h(x)$ where $x\in h\Z^d$. 
Due to \eqref{scaled-grad-lip} we have that $w_{h,e}$ is $C_\rho$-Lipschitz on $E_\rho(h)$.
Next, in analogy with \eqref{ext-1} define
\begin{equation}\label{ext-2}
 W_{h,e}^\rho(x) = \inf\limits_{\xi \in E_{\rho}(h)} \left( w_{h,e}(\xi) + C_{\rho} |x-\xi| \right) , \qquad x\in \R^d,
\end{equation}
which provides a Lipschitz extension for $w_{h,e}$ in the entire space $\R^d$.
What we also obtain, in view of  definition \eqref{ext-2}, is that
\begin{equation}\label{w-h-est}
|W_{h,e}^\rho (x)  - \nabla_e^h U_h (x) | = | W_{h,e}^\rho (x) - w_{h,e} (\xi) | \leq C_{\rho} h,
\end{equation}
where $\xi \in E_{2 \rho }(h)$ and $x \in \mathrm{C}_h(\xi)$. Note that here  we have used the fact that the extension $U_h$ preserves the discrete derivatives.

\bigskip
By construction, we have that $\{ U_h^\rho \}$ and $ \{ W_{h,e}^\rho \} $ form a family of equicontinuous and uniformly bounded functions with respect to $h$,
and hence, applying Arzel\`{a}-Ascoli we extract a subsequence $h_k\to 0$ as $k\to \infty$, such that
$U_{h_k}^\rho \to U_0^\rho$ and $W_{h_k,e}^\rho \to W_{0,e}^\rho$ locally uniformly in $\R^d$ as $k\to \infty$,
where $U_0^\rho$, and $W_{0,e}^\rho$ for any $e$, are
Lipschitz functions defined on $\R^d$. We now show that
\begin{equation}\label{Laplace-U}
 \Delta U_0^\rho  = 2d m \mathbb{I}_{ \{ 0<U_0^\rho <1/m  \} }  \text{ in } \{ U_0^\rho>0 \}\setminus \overline{B_{2\rho}},
\end{equation}
and 
\begin{equation}\label{gradient-U}
 \nabla_e U_0^\rho = W_{0,e}^\rho \text{ in } \{ U_0^\rho>0 \}\setminus \overline{B_{2\rho}},
\end{equation}
where $\Delta$ and $\nabla_e$ are the  Laplacian and $e$-directional derivative respectively.

\vspace{0.2cm}
We start with \eqref{Laplace-U}. Following the definition of a weak solution
to \eqref{Laplace-U}, we need to check that
\begin{equation}\label{Laplace-U1}
\int U_0^\rho \Delta \varphi dx = 2d m \int \mathbb{I}_{ \{ 0<U_0^\rho <1/m  \} } \varphi dx,
\end{equation}
for any $\varphi \in C_0^\infty (B)$ where $B\subset \{U_0^\rho >0 \} \setminus \overline{B_{2\rho}}$
is a ball. 
Using discrete integration by parts, we have
\begin{multline}\label{long-f}
\int U_0^\rho \Delta \varphi dx = 2d \lim_{k\to \infty} \int U_{h_k}^\rho \Delta^{h_k} \varphi dx =
2d \lim_{k\to \infty} \int  \Delta^{h_k} U_{h_k}^\rho  \varphi dx = \\
2d m \lim_{k\to \infty} \int  \mathbb{I}_{ \{ 0< U_{h_k}\leq 1/m \} }  \varphi dx,
\end{multline}
and comparing with \eqref{Laplace-U1}, we see that what needs to be proved is
that
\begin{equation}\label{weak-conv1}
\mathbb{I}_{ \{ 0< U_{h_k}\leq 1/m \} } \to \mathbb{I}_{ \{ 0< U_0^\rho <  1/m \} } \ \
 \text{ weak}^{\ast} \text{ in } \{ U_0^\rho>0 \}\setminus \overline{B_{2\rho}}.
\end{equation}
Observe, that for any $x\in \{ U_0^\rho>0 \}\setminus \overline{B_{2\rho}}$ where
$U_0^\rho (x) \neq \frac 1m$, the convergence in \eqref{weak-conv1} is pointwise
in view of the convergence of $U_{h_k}^\rho$ and estimate \eqref{u-h-U-h-est}.
Now, if $U_0^\rho(x) = \frac 1m$ and $x$ is an interior point of the set $\{ U_0^\rho = \frac 1m \}$,
then we have the convergence in \eqref{weak-conv1} in the neighbourhood of $x$,
thanks to \eqref{long-f} and the fact that $\Delta U_0^\rho =0$ in a neighbourhood of $x$.
Thus, in \eqref{weak-conv1}, it is left to cover the case when
$x\in \partial \{ U_0^\rho = \frac 1m \}$.
We will show that this set has measure 0, and hence can be ignored in \eqref{weak-conv1}.
Indeed, observe that $U_0^\rho$ trivially inherits, in the same form but now
on $\R^d$, the directional monotonicity of odometers
established in Theorem \ref{Thm-monotonicity}.
But then, following Remark \ref{rem-graph} (see also the proof of Proposition \ref{prop-size-of-mass-FREE-region},
and \cite[Theorem 5.3 (v)]{AS}), we get that the set
$\partial \{U_0^\rho = \frac 1m \}$ is locally a graph of a Lipschitz function,
and is hence rectifiable. In particular, the measure of $\partial \{U_0^\rho = \frac 1m\}$
is 0, which completes the proof of \eqref{weak-conv1}, and hence  \eqref{weak-conv1} follows,
which completes the proof of \eqref{Laplace-U}.

\vspace{0.2cm}
We next proceed to the proof of \eqref{gradient-U}. Fix any unit vector $e\in \Z^d$ and any point $x_0 \in \{ U_0^\rho>0 \}\setminus \overline{B_{2\rho}}$.
Since $U_0^\rho$ solves \eqref{Laplace-U}, then it is $C^1$ in the neighbourhood of $x_0$, and hence
it is enough to show that for any $\varphi\in C_0^\infty$ having a support in a small neighbourhood of $x_0$, one has
$$
\int W_{0,e}^\rho \varphi dx = \int \nabla_e U_0^\rho \varphi dx= - \int U_0^\rho \nabla_e \varphi dx,
$$
where the second equality is simply integration by parts, and what needs to be proved is the first one.
For that one, we observe
\begin{multline}
\int W_{0,e}^\rho \varphi dx = \lim\limits_{k\to \infty} \int W_{h_k,e}^\rho (x) \varphi(x) dx = \\
\lim\limits_{k\to \infty} \int \nabla_e^{h_k} U_{h_k}  (x) \varphi(x) dx = 
\lim\limits_{k\to \infty} \int U_{h_k} (x) \nabla_{-e}^{h_k} \varphi(x) dx  = \\
\int U_0^\rho (x) \nabla_{-e} \varphi(x) dx = -\int U_0^\rho (x) \nabla_e \varphi(x) dx,
\end{multline}
and hence the equality in \eqref{gradient-U}.

\bigskip

Due to \eqref{scaled-grad-lip} the function $W_{0,\rho}$ is $C^{0,1}$ in  the neighbourhood of $\partial \{U_0^\rho>0\}$, and is vanishing everywhere
outside the support of $U_0^\rho$, and hence from \eqref{gradient-U} we obtain 
\begin{equation}\label{grad-0-bdry}
|\nabla U_0^\rho |=0  \text{ on } \partial \{U_0^\rho >0 \} .
\end{equation}

\bigskip

Finally, applying diagonal argument
as $\rho\to 0$,  we conclude the existence of a subsequence $h_k\to 0$ as $k\to \infty$,
and a function $U_0$ such that
\begin{itemize}
\item[(a)] $U_0 \geq 0 $ is compactly supported, and is $C^{1,1}$ in $\R^d$ away from any open neighbourhood of the origin;

\smallskip
\item[(b)] $U_{h_k} \to U_0$ uniformly outside any open neighbourhood of the origin

\smallskip
\item[(c)] $\Delta U_0 = 2d m \mathbb{I}_{\{ 0<U_0<1/m \}}  -\delta_0$ in $\R^d$.
\end{itemize}
Note, that in (c) we get the equation in $\R^d$ instead of only $\{U_0>0\}$,
as $|\nabla U_0| = 0$ on $\partial \{U_0>0\}$. The only thing which needs a proof, is
the equation in (c) near the origin. That, again, can be handled as in \cite[Theorem 5.3]{AS}.
Namely, for each $h>0$ let $\Phi_h $ be the fundamental solution to $\Delta^h$ in $h\Z^d$.
Then, $\Delta^h (u_h  - \Phi_h) =0$ in $B_r$, where $r>0$ is fixed small enough such that
$B_r\subset \{u_h > 1/m \}$. The existence of such $r$ follows from Proposition \ref{prop-size-of-mass-FREE-region}.
Applying DMP we obtain, that $|u_h  - \Phi_h|\leq C$ in $B_r$, and hence the limit
$|U_0 - \Phi_0|\leq C$ in $B_r\setminus \{0\}$ and solves $\Delta ( U_0 - \Phi_0 )=0$
in $B_r\setminus \{0\}$. It follows that the origin is a removable singularity 
for the function harmonic $U_0 - \Phi_0$, implying the equality in (c) near the origin.

\bigskip

We now apply Lemma \ref{lem-unique} from Appendix, which states that there
is a unique solution to (c), and that solution is spherically symmetric.
With this at hand, we conclude the existence of the scaling limit of the sequence $U_h$,
without passing to a subsequence, since we get that any sequence of scales, contains a
subsequence, along which $U_h$ converges to the same limit.

The proof of the theorem is now complete.
\end{proof}

\begin{remark}
Observe, that convergence in \eqref{weak-conv1} implies that the set of a visited sites
of the sandpile, after scaling, converges to the support of the limiting odometer.
In particular, the family of odometers $\{u_h\}$ do not degenerate as $h\to 0$.
\end{remark}

Now, a simple compactness argument allows us to take limits when $m\to \infty$.
This in particular, gives a sandpile dynamics, which redistributes all mass on the free boundary,
which has a circular shape.

\begin{theorem}\label{Thm-scaling-limit-inf-M}{\normalfont(Scaling limit as $m\to \infty$)}
Let $u_{n,m}$  be  the odometer of the sandpile with initial distribution $n \delta_0$, and threshold $1<m<n$. Then there is a slowly increasing  function $F:\R_+ \to \R_+$, with $F(+\infty) = +\infty $, such that 
the odometers $u_{n,F(n)}$,
 after rescaling by $n^{-2/d}$, converge to a spherically
symmetric function, as $ n \to \infty$.
\end{theorem}

\begin{proof}
Thanks to Theorem \ref{Thm-scaling-limit-fixed-M}, for each $m$ fixed, there exists a scaling limit
for $\{u_{n,m}\}$ as $n\to \infty$. Let $u_{0,m}$ be this limit, which we know is spherically symmetric.
Next, according to Lemma \ref{lem-rad-conv}, we get that $u_{0,m}$ converges uniformly, away from the origin, to some $u_0$, a spherically symmetric
function.
Finally, for each given $m$ we choose $n=n(m)$ large enough, such that
$$
|n^{-2/d} u_{n,m} - u_{0,m} | \leq \frac{1}{m}, \text{ when } |x|\geq \frac 1m.
$$
We may obviously   choose $n= n(m)$ as increasing function of $m$.
Taking $m$ to infinity, and using the convergence of $u_{0,m}$ to $u_0$, we have a final spherical configuration. Now reverting the relation $m \mapsto n$, we get a function $n\mapsto F(n)$ which satisfies the conditions in the theorem. This completes the proof.
\end{proof}

\appendix

\section{Uniqueness of the scaling limit}\label{sec-unique}

For any $m, k >0$ fixed,  let $u$ (being compactly 
supported)\footnote{This assumption is not any restriction to our theory, as the support of odometer functions are compact. 
It should be remarked that one  can show existence of  unbounded solutions to this particular free boundary problem.} solve 
\begin{equation}\label{contin-pde}
\Delta u = m \mathbb{I}_{\{0<u< k\}} -   \delta_0 .
\end{equation}
We want to prove that such a $u$ is unique and spherically  symmetric. 
To see this, we show that if $u$ is not spherically symmetric, then we may generate two distinct spherically
symmetric solutions by starting from infimum/supremum over all rotations of $u$.
Then, we prove that given the solution is spherically symmetric, then it is uniquely determined
by coefficients $m$ and $k$. These two components put together settle the uniqueness of solutions to
\eqref{contin-pde}.

\bigskip

We shall now take the smallest super-solution solution and we only need to check that the infimum in this class  does not degenerate to zero and that it solves the problem. Obviously we may only consider spherically symmetric super-solutions with shrinking  radius of support. Hence integration by parts 
(using that $u_*$ vanishes outside a large ball) implies $Volume(\{ 0< u_* < k\}) \geq 1$.
This bound, coupled with the spherical symmetry of $u_*$ implies
that the support of $u_*$ must contain a ball of a fixed radius, in particular 
the support cannot shrink too much. 
So it remains to show that the smallest super-solution is actually a solution. Now in the set $\{u_* > k\}$, we may always replace $u_*$ by $h$ which is the solution of 
$\Delta h = -\delta_0$, with boundary values $k$ on $\partial \{u_* > k\}$. Also in the set 
$\{ u_* < k\}$ we make a replacement with solution to $\Delta h = k\mathbb{I}_{ \{h>0 \}}$ and the corresponding boundary values. Hence we may assume $\Delta u_* = m \mathbb{I}_{\{0<u< k\}} -   \delta_0  - \mu_* $, where $support (\mu_*)$ is on the sphere $ \partial \{u_* > k\}  $.  Next let $v_*$ solve $\Delta v_* = -\delta_0 + m\mathbb{I}_{\{0< u_* < k\}}$ and $v_* = 0 $ on the boundary of the support of $u_*$.
Then by comparison principle $v_* \leq u_*$, and hence $\Delta v_* \leq - \delta_0 + m\mathbb{I}_{ \{0< v_* < k\}} $, implying that $v_*$ is also a super-solution to our problem. Now $u_*$ being infimum in the class implies $v_*=u_*$, and hence $u_*$ solves \eqref{contin-pde}.

\bigskip 

A similar argument works by taking the largest sub-solution, among all solutions with bounded support. 
Here again we may consider spherically symmetric sub-solutions. We need to show that the largest sub-solution in the class stays bounded.  As in previous case we may replace the maximizing sequence of sub-solutions  $u_j^*$ with exact solutions in the sets 
$\{ u_j^* > k \}$ and $\{u_j^* < k \}$. To see that the support of such sub-solutions  should stay uniformly bounded we argue as follows.  Let $R_j^*$ be the radius of the support of $u^*_j$, and $R_j$ be such that $u_j^* = k$ on $|x|=R_j$.   Using sub-solutions properties, we can conclude that 
$Volume(\{ 0< u^*_j < k\}) \leq 1$ (this is the reverse of the previous inequality). Hence 
$R_j^* \approx  R_j + c_dR_j^{1-d}$, for some dimensional constant $c_d >0$.

Let now $F$ denote the fundamental solution, and set $F_j =\max (F - a_d(R_j^*)^{2-d}, 0 ) $, for 
$d \geq 3$ and $F_j =  a_2 \max (\log (R_j^*/|x|), 0 )$,
where $a_d$ is a normalization constant for $d\geq 2$.
Since 
$$\Delta F_j = -\delta_0 \leq -\delta_0 + m \mathbb{I}_{\{0< u^*_j < 1/k \}} \leq \Delta u^*_j, \qquad 
\hbox{in } B_{R_j^*} 
$$
and 
$$
 F_j > 0 = u^*_j \quad \hbox{on } \partial B_{R^*_j},$$
we can apply  comparison principle  to deduce $F_j    \geq u^*_j$  on  $B_{R_j^*}$.
But on the other hand  for $|x| = R_j \approx  R_j^*  - c_dR_j^{1-d}  $
we have   
 $$F_j  (x) = aR_j^{2-d}  - a(R_j^*)^{2-d} <  1/k = u_j^* (x)$$ 
once  $ R_j^*   \approx  R_j + c_dR_j^{1-d}$  is large  enough. Hence a contradiction. 
This implies that the support of sub-solutions should stay bounded.
Now taking the largest sub-solution, and using the fact that their supports are uniformly
bounded, we get yet another solution to \eqref{contin-pde}.

\bigskip

The conclusion is that if \eqref{contin-pde} has a solution
with bounded support, which is not spherically symmetric,
then we can produce two distinct spherically symmetric solutions.
Hence to show uniqueness of solutions to \eqref{contin-pde}
with bounded support, it suffices to show that there is only
one spherically symmetric solution to \eqref{contin-pde}
having bounded support. This we establish in the next lemma.

\begin{lem}\label{lem-unique}
For any $m,k>0$ there is a unique spherically symmetric solution to \eqref{contin-pde}
having bounded support.
\end{lem}

\begin{proof}
We shall give a computational proof, which is elementary, but tedious. 
First we consider the case when $d>2$.
Let $\omega_d>0$ be the normalising constant of the Green's kernel\footnote{It is well-known
that $\omega_d = (d(d-2) |B_1|)^{-1}$ where $|B_1| $ is the volume
of the unit ball in $\R^d$, but the actual value of this constant will be of no relevance to our proof.}.
The spherical symmetry of $u$ implies that it should be of the form
$$
 u(x) =  \begin{cases}  a_1 + \omega_d |x|^{2-d} , &\text{if $0<|x|\leq r_1$}, \\ 
  a_2 + a_3|x|^{2-d} + \frac{m}{2d} |x|^2  , &\text{if } r_1<|x|\leq r_2,    \end{cases}
 $$
where $a_i$ and $r_j$ are constants.
The aim is to show that these constant are uniquely determined from \eqref{contin-pde},
which we do next.

By a straightforward computation we have
\begin{eqnarray}
\label{aa1} a_2 + a_3 r_2^{2-d} + \frac{m}{2d} r_2^2 =0 \\
\label{aa2} a_3 (2-d) r_2^{-d} + \frac md =0 \\
\label{aa3} a_1 + \omega_d r_1^{2-d} = k \\
\label{aa4} a_2 + a_3 r_1^{2-d} +\frac{m}{2d} r_1^2 = k \\
\label{aa5} \omega_d (2-d) r_1^{-d} = a_3 (2-d) r_1^{-d} + \frac md,
\end{eqnarray}
where \eqref{aa1} is due to the condition $u=0$ on $|x|=r_2$, \eqref{aa2} comes from $|\nabla u|=0$ on $|x|=r_2$,
\eqref{aa3} is in view of $u=k$ on $|x|=r_1$ and \eqref{aa4} is the continuity of $u$ on $|x|=r_1$,
finally \eqref{aa5} is the continuity of the gradient on $|x|=r_1$.

From \eqref{aa5} and \eqref{aa2} we have 
\begin{equation}\label{aa5-1}
a_3= \omega_d +  \frac{m}{d(d-2)} r_1^d = \frac{m}{d(d-2)} r_2^d,
\end{equation}
and hence 
\begin{equation}\label{rr2}
r_2= \left[ r_1^d + \omega_d   \frac{d (d-2)}{m}   \right]^{1/d}.
\end{equation}
Subtracting \eqref{aa2} from \eqref{aa4}, to eliminate $a_2$, we obtain
\begin{equation}
\frac{m}{d(d-2)} r_1^{2-d} \left[ r_1^d + \omega_d   \frac{d (d-2)}{m}   \right]
+\frac{m}{2d} r_1^d - \frac{m}{2(d-2)} r_2^2 = k.
\end{equation}
Plugging the value of $r_2$ from \eqref{rr2}, we get
\begin{equation}\label{r1-eq}
\frac{m}{2(d-2)} r_1^2 + \omega_d r_1^{2-d} -
\frac{m}{2(d-2)} \left[ r_1^d + \omega_d \frac{d(d-2)}{m}  \right]^\frac 2d =k.
\end{equation}

To complete the proof we need to see that \eqref{r1-eq} has a unique solution when $0<r_1<\infty$.
Rearranging \eqref{r1-eq}, consider the function
\begin{equation}\label{r-1-for-d-large}
f(x)  = k - \frac{\omega_d}{x^{d-2}} - \frac{m}{2(d-2)} x^2 +
\frac{m}{2(d-2)} \left[ x^d + \omega_d \frac{d(d-2)}{m}  \right]^{\frac 2d},
\end{equation}
where $0<x<\infty$. We need to show that $f$ has a unique zero in $(0,\infty)$.

Observe, that $f(0+)= -\infty$, and $f(+\infty) = k >0$, and hence the continuity of $f$ gives the existence of
a zero for $f$. We are left with establishing the uniqueness of this zero.
To this end, computing the derivative of $f$ we get
\begin{multline*}
f'(x) = (d-2) \frac{\omega_d}{x^{d-1}} + x \frac{m}{d-2} 
\left[ \left(  1+\frac{\omega_d}{x^d} \frac{d(d-2)}{m}  \right)^{\frac{2-d}{d}} -1  \right] =\\
x \left\{ (d-2) \frac{\omega_d}{x^d}   + \frac{m}{d-2} \left[ \left(  1+\frac{\omega_d}{x^d} \frac{d(d-2)}{m}
 \right)^{ \frac{2-d}{d} } -1 \right]   \right\} .
\end{multline*}
We want to show that $f'(x)>0$ everywhere, which will complete the proof.
Ignoring the $x$ in front of $f'$ in the last expression and denoting $y: = (d-1)\omega_d x^{-d} $, we obtain
that $f'(x) = x F(y)$, where
$$
F(y)  = y + \frac{m}{d-2} \left[  \left(  1+\frac dm y \right)^{ \frac{2-d}{d}} -1 \right],
$$
and hence it is enough to show that $F>0 $ everywhere. The latter follows
easily from the fact that $F(0) =0 $ and $F'>0$ in $(0,\infty)$.
We conclude the proof of uniqueness in the case when $d\geq 3$.

\bigskip

We now treat the case of $d=2$.
Set $\omega_2=-1/(2\pi)$ which is the normalising constant for the Green's kernel.
Again, the exact value of $\omega_2$ is not important for the proof,
and we will only use that $\omega<0$.
Relying on the spherical symmetry of $u$, as above, we see that $u$ has to be of the form
$$
 u(x) =  \begin{cases}  a_1 + \omega_2 \log |x| , &\text{if $0<|x|\leq r_1$}, \\ 
  a_2 + a_3 \log|x|  + \frac{m}{4} |x|^2  , &\text{if } r_1<|x|\leq r_2,    \end{cases}
 $$
where $a_i$ and $r_j$ are unknown constants. The goal is to show that these constants
are being determined uniquely given the properties of $u$. As a above, a direct computation leads to
\begin{eqnarray}
\label{bb1} a_2 + a_3 \log r_2  + \frac{m}{4} r_2^2 =0 \\
\label{bb2} a_3  r_2^{-2} + \frac m2 =0 \\
\label{bb3} a_1 + \omega_2 \log r_1  = k \\
\label{bb4} a_2 + a_3 \log r_1  + \frac{m}{4} r_1^2 = k \\
\label{bb5} \omega_2  r_1^{-2} = a_3   r_1^{-2} + \frac m2,
\end{eqnarray}
where \eqref{bb1} is due to the condition $u=0$ on $|x|=r_2$, \eqref{bb2} comes from $|\nabla u|=0$ on $|x|=r_2$,
\eqref{bb3} is in view of $u=k$ on $|x|=r_1$ and \eqref{bb4} is the continuity of $u$ on $|x|=r_1$,
finally \eqref{bb5} is the continuity of the gradient on $|x|=r_1$.

Using \eqref{bb2} we get $a_3= -\frac m2 r_2^2 $, which together with \eqref{bb5}
implies
$$
r_2= \left( r_1^2 -\frac 2m \omega_2 \right)^{\frac 12}.
$$
Subtracting \eqref{bb4} from \eqref{bb1} we obtain
$$
a_3 \log \frac{r_1}{r_2} + \frac m4 (r_1^2 - r_2^2 ) = k.
$$
From the above relation between $r_1$ and $r_2$ we have
$$
 k - \frac{\omega_2}{2} = a_3 \log \frac{r_1}{r_2} = - \frac{a_3}{2}  \log \left( \frac{r_2}{r_1} \right)^2
$$
Plugging the values of $a_3$ and $r_2$ we get
\begin{equation}\label{r-1-for-d2}
k- \frac{\omega_2}{2} = \frac m4 \left( r_1^2 - \frac 2m \omega_2 \right)
\log \left(1- \frac 2m \frac{\omega_2}{r_1^2} \right),
\end{equation}
and thus need to show that the last equation has a unique solution in $(0, \infty)$.
To this end, set $y := -\frac{2}{m} \frac{\omega_2}{r_1^2}$.
With this notation, we need to prove that the function
$$
F(y) = -\frac 12 \omega_2 \left(1+ \frac 1y \right) \log (1+y) - k +\frac{\omega_2}{2},
$$
has a unique zero in the interval $(0,\infty)$.
Observe, that $F(0+) = - k + \frac{\omega_2}{2}<0$ as $\omega_2<0$,
and $F(+\infty)=+\infty$, hence the existence.
For uniqueness, computing the derivative of $F$, we see
$$
F'(y) = -\frac{\omega_2}{2y} \left(1-\frac{\log(1+y)}{y} \right),
$$
which is always positive in the range $y\in(0, \infty)$, and the uniqueness follows.

The proof of the lemma is now complete.
\end{proof}

As we saw in Section \ref{sec-scaling-lim}, the scaling limit of the sandpile generated by a single
source at the origin, is determined by the following PDE
\begin{equation}
 \Delta u = 2d m \mathbb{I}_{ \{ 0<u<1/m } \} - 2d \delta_0 \text{ in } \R^d,
\end{equation}
where the support of $u$ bounded and contains a ball
of some fixed radius, in both cases uniformly with respect to $m$.
From Lemma \ref{lem-unique} we have that $u$ is spherically symmetric,
moreover, if we let $x_m$ be the radius of the ball $\{ u\leq \frac 1m \}$,
then following \eqref{r-1-for-d-large} and \eqref{r-1-for-d2} we have that $x_m$ is the unique
zero in the range $(0,\infty)$ of the function 
\begin{equation}\label{f1}
 f(x)  = 1 - 2d m \frac{\omega_d   }{x^{d-2}} - m^2 \frac{d}{d-2} x^2 + m^2 \frac{d}{d-2} \left[  x^d + \omega_d \frac{d(d-2)}{m}  \right]^{\frac 2d},
\end{equation}
for dimension $d>2$, and for dimension 2, rearranging \eqref{r-1-for-d2}, and plugging the value of $\omega_2 = -\frac{1}{2\pi} $,
we see that $x_m$ is the unique solution in the range $(0,\infty)$ of the equation
\begin{equation}\label{f2}
\frac{1}{m^2} + \frac{1}{\pi m} =   \left( x^2 + \frac{1}{\pi m} \right) \log \left(1+\frac{1}{\pi m x^2} \right).
\end{equation}
Our final result shows that the limiting shapes of the sandpiles, corresponding to each $m$
converge, as $m\to \infty$. For this it will be enough to prove that the radii $x_m$ converge.
This we do next; the proof is elementary and straightforward.

\begin{lem}\label{lem-rad-conv}
The radii $x_m$ converge as $m\to \infty$.
\end{lem}
\begin{proof}
Due to their relation with the sandpile, the set $\{x_m\}_{m\geq 1}$ is bounded away from zero and infinity,
as we outlined above. We first consider the case when $d>2$. Rearranging the last two terms in \eqref{f1}
we get, for $x=x_m$, that
\begin{multline*}
f(x) = 1 - 2d m \frac{\omega_d   }{x^{d-2}}  + m^2 \frac{d}{d-2} x^2 \left[  \left(  1 + \frac{\omega_d}{x^d} \frac{d(d-2)}{m }  \right)^{\frac 2d}  -1  \right] = \\
1 - 2d m \frac{\omega_d   }{x^{d-2}}  + 
m^2 \frac{d}{d-2} x^2 \left[ \frac 2d \frac{\omega_d}{x^d} \frac{d(d-2)}{m} + \frac 12 \frac 2d \left( \frac 2d -1 \right) \omega_d^2 \frac{( d(d-2))^2}{x^{2d} m^2} + \mathrm{O}(m^{-3})   \right]  = \\
1- x^{2-2d} \omega_d^2 (d-2)^2 d + \mathrm{O}(m^{-1}), \text{ as } m\to \infty,
\end{multline*}
where we have used the fact that $\{x_m\}$ is bounded away from 0 and infinity.
Since $f(x_m) =0$, from the last expression, taking $m\to \infty$, we see that
$$
\lim\limits_{m\to \infty} x_m = (d(d-2)^2 \omega_d^2 )^{\frac{1}{2(d-2)}}.
$$

\bigskip

We now settle the case of $d=2$. Defining $y= \pi m x^2$ in \eqref{f2},
we get
$$
1+\frac{\pi}{m} = (y+1) \log \left(1+\frac 1y \right).
$$
Now, using that $x_m$, the solution to \eqref{f2} is bounded from 0 and infinity,
we do asymptotic expansion in the last expression, when $m \to \infty$ (and hence $y\to \infty$), which gives
$$
1+\frac{\pi}{m} = 1+ \frac{1}{2y} + \mathrm{O} (y^{-2}) \text{ as } m\to \infty,
$$
and hence, as $y=\pi m x_m^2$, we obtain that 
\begin{equation}\label{m2}
\lim\limits_{m\to\infty} x_m = \frac{1}{\sqrt{2}} \frac{1}{\pi}.
\end{equation}

The proof of the lemma is now complete.
\end{proof}

\end{document}